\documentclass{article}

\usepackage[utf8]{inputenc}     
\usepackage[english]{babel}     
\usepackage[a4paper, left=1.5in, right=1.5in, top=1.5in, bottom=1.5in]{geometry}

\usepackage{color,hyperref}
\usepackage{titlesec}
\titleformat{\paragraph}
{\normalfont\normalsize\bfseries}{\theparagraph}{1em}{}
\titlespacing*{\paragraph}
{0pt}{3.25ex plus 1ex minus .2ex}{1.5ex plus .2ex}

\usepackage{tikz}
\usetikzlibrary{arrows,positioning,shapes.geometric}
\usepackage[useregional]{datetime2}             
             
\usepackage{graphicx}           
\usepackage{amsmath,amssymb}    
\usepackage{amsthm}             
\usepackage{amsmath}
\usepackage{mathtools}          
\usepackage{enumitem}           
\usepackage{listings}           
\usepackage{todonotes}          
\usepackage{hyperref}  
\usepackage{bbold}

\newtheorem{theorem}{Theorem}[section]

\newtheorem{prop}[theorem]{Proposition}
\newtheorem{lemma}[theorem]{Lemma}

\newtheorem{remark}{Remark}
\title{Random multiplicative functions and typical size of character in short intervals}
\author{Rachid Caich }
\date{\today}

\begin{document}

\maketitle
\begin{abstract}
We examine the conditions under which the sum of random multiplicative functions in short intervals, given by $\sum_{x<n \leqslant x+y} f(n)$, exhibits the phenomenon of \textit{better than square-root cancellation}. We establish that the point at which the square-root cancellation diminishes significantly is approximately when the ratio $\log\big(\frac{x}{y}\big)$ is around $\sqrt{\log\log x}$. By modeling characters by random multiplicative functions, we give a sharp bound of $\frac{1}{r-1}\sum_{\chi \!\!\!\mod r} \big|\sum_{x<n\leqslant x+y}\chi(n)\big|$, where $r$ is a large prime and $x+y\leqslant r $. This extends the result of Harper \cite{Harper_charac}.\\

\end{abstract}

\section{Introduction}
The aim of this article is to investigate the typical size of random multiplicative functions in short intervals. Recent years have witnessed significant interest directed towards the study of random multiplicative functions. Within the realm of number theory and probability, two prominent models of random multiplicative functions have gained considerable attention. 

\noindent Let $\mathcal{P}$ be the set of prime numbers. One of these models is referred to as the \textit{Steinhaus random multiplicative function}. It is constructed by defining $(f(p))_{p\in \mathcal{P}}$ as a sequence of independent Steinhaus random variables, which are uniformly distributed on the unit circle $|z|=1$. Subsequently, $f(n)$ is defined for each $n \in \mathbb{N}$ as follows: $f(n):=\prod_{p^{a}||n}f(p)^{a}$, where $p^{a}||n$ indicates that $p^{a}$ is the highest power of the prime $p$ dividing $n$.

\noindent Another model is the \textit{Rademacher multiplicative function}, obtained by letting $(f(p))_{p\in \mathcal{P}}$ be independent Rademacher random variables, which take values $\pm 1$ with equal probability of $\tfrac{1}{2}$ each, and we set
$$
f(n)=\left\{
  \begin{array}{@{}ll@{}}
    \prod_{p |n} f(n) , & \text{if}\ n \text{ is squarefree} \\
    0, & \text{otherwise.}
  \end{array}\right.
$$
The Rademacher model was introduced by Wintner \cite{Wintner} in 1944 as a heuristic model for the Möbius function $\mu$ (see the introduction \cite{Tenenbaum}). With a slight modification, one can derive a probabilistic model for a real primitive Dirichlet character (see Granville--Soundararajan \cite{GranvilleSound2}). The Steinhaus random multiplicative functions serve as models for randomly chosen Dirichlet characters $\chi$ or continuous characters $n \mapsto n^{it}$; for more details, refer to section 2 of Granville--Soundararajan \cite{GranvilleSound} and Harper \cite{Harper_charac}. 

\noindent Given such a random multiplicative function, an important goal
is to understand the partial sums $\sum_{n \leqslant x} f(n)$. Since $\mathbb{E}\big[f(n)\overline{f(m)}\big]=\mathbb{1}_{n=m}$, we have then 
$$
\mathbb{E}\bigg[\bigg|\sum_{n \leqslant x}f(n)\bigg|^2\bigg] \asymp x.
$$
One of the most remarkable results in the study of random multiplicative functions can arguably be attributed to Harper's work \cite{Harper}, which is: when $x \to +\infty$
$$
\mathbb{E}\bigg[ \bigg| \sum_{n\leqslant x}f(n)\bigg| \bigg] \asymp \frac{\sqrt{x}}{(\log_2 x)^{1/4}}.
$$
Here and in the sequel $\log_k$ denotes the k-fold iterated logarithm. Harper's theorem established Helson's conjecture \cite{helson_conj}, which says that the partial sums of random multiplicative functions generally manifest more than square-root cancellation than what is expected from square-roots. The veracity of Helson's conjecture appeared considerably uncertain during that period. In fact, prior research conducted by Harper, Nikeghbali, and Radziwiłł \cite{HNR} had suggested the opposite  Helson's conjecture.

\noindent More recently, Harper \cite{Harper_charac} proved an interesting
result about character sums, in fact, he proved, if $x$ and $r/x$ tend to infinity with $r$, then 
$$
\frac{1}{r-1}\sum_{\chi\!\!\!\!\! \mod r} \bigg|\sum_{n\leqslant x}\chi(n)\bigg| = o(\sqrt{x}).
$$
The proof relies on establishing a link between the sum of characters and the sum of Steinhaus multiplicative functions. After this connection is established, the phenomenon of \textit{better than square-root cancellation} arises.

\noindent The aim of this paper is to study the typical size of random multiplicative functions in short intervals $]x,x+y]$. In suitable ranges, we shall establish that the sum of a random multiplicative function has more than square-root cancellation. In a prior work by Chatterjee and Soundararajan \cite{sound_chat}, they established, in the case of Rademacher functions, that when $y=o\big(\frac{x}{\log x}\big)$, the sum $\sum_{x<n\leqslant x+y}f(n)$ satisfies to a central limit theorem. More recently, Soundararajan and Xu \cite{Sound_Xu} improved this result \cite{sound_chat} to a range $y\leqslant \frac{x}{(\log x)^{\alpha+\varepsilon}}$ where $\alpha=2\log 2 -1$ for any $\varepsilon>0$. 
Thus, in this range, one can easily deduce that when $y$ tends to infinity, we have for any $0 \leqslant  q\leqslant 1$
$$
\mathbb{E}\bigg[\frac{1}{(A(x,y))^{q}}\bigg|\sum_{x<n\leqslant x+y}f(n)\bigg|^{2q}\bigg]\sim \frac{\Gamma(\frac{2q+1}{2})2^{\frac{2q+1}{2}}}{\sqrt{2\pi}}
$$
where $A(x,y):= \mathbb{E}\bigg[\bigg|\sum_{x<n\leqslant x+y}f(n)\bigg|^{2}\bigg]$ in particular $A(x,y)$ is equal to $y +O(1)$ in Steinhaus case. In Rademacher case $A(x,y)$ is equal to $ |\{ n\in [x,x+y]: \,\, n \,\,\,\text{ square-free integer} \}|$. Alternatively, according to Harper's \cite{Harper}, it is proven that the probability distribution $\sum_{x<n\leqslant x+y}f(n)$ deviates from a Gaussian distribution for $y \geqslant \frac{x}{(\log_2 x)^{1/2-\varepsilon }}$.

\noindent For the typical size problem, Xu \cite{Xu} had investigated the phenomena of the better than square-root cancellation when the sum is restricted to $n\leqslant x$ of $R$ rough numbers. For $\mathcal{A}_R(x)=\{ n\leqslant x : p|n \implies p >R\}$, he showed that  $$
\mathbb{E}\bigg[ \frac{1}{\sqrt{\mathcal{A}_R(x)}}\bigg|\sum_{n\in \mathcal{A}_R(x)}f(n)\bigg|\bigg]=o(1)
$$
when $\log_2 R =o(\sqrt{\log_2 x})$ and when $\log_2 R \gg \sqrt{\log_2 x}$
and 
 $$
\mathbb{E}\bigg[\frac{1}{\sqrt{\mathcal{A}_R(x)}}\bigg|\sum_{n\in \mathcal{A}_R(x)}f(n)\bigg|\bigg]\asymp 1,
$$
which means that $ R \approx \sqrt{\log\log x}$ is the threshold.

\noindent In our case : $\sum_{x< n\leqslant x+y}f(n)$, we anticipate observing the following behavior at certain thresholds. When the value of $y$ is sufficiently large, the range under consideration approximates $[x,2x]$ and since $\sum_{x< n\leqslant 2x} f(n) = \sum_{n\leqslant 2x}f(n) -\sum_{n\leqslant x}f(n)$, the better than square-root cancellation appears as done by Harper \cite{Harper}.  On the other hand, when $y$ takes on sufficiently small values, the presence of weak dependence might even lead to the emergence of a central limit theorem. This phenomenon is explored, as mentioned before, in the works of Chatterjee and Soundararajan as \cite{sound_chat} and Soundararajan and Xu \cite{Sound_Xu}. Our results in Theorem~\ref{theoreme_principal_1} and Theorem \ref{theoreme_principal_2} provide a conclusive response to this inquiry. We establish that the threshold for this behavior is characterized by $\log\big(\frac{x}{y}\big) \approx \sqrt{\log_2 x}$.
Define 
\begin{equation}\label{def_M}
    M(x,y;f):= \sum_{x<n\leqslant x+y}f(n).
\end{equation}
The main results of this paper are the following three theorems.
\begin{theorem}\label{theoreme_principal_1}
Let $f$ be a Rademacher or Steinhaus random multiplicative function. Let $y\leqslant x $ large numbers, we define $\theta$ by $\frac{x}{y}=\log^{\theta} x$. We have uniformly for $\theta > 0$ and $ 0\leqslant q \leqslant 1$
\begin{equation*}
    \mathbb{E}\bigg[ \big| M(x,y;f)\big|^{2q}\bigg] \ll \bigg( y \min\bigg\{ 1, \theta \sqrt{\log_2 x} + \frac{1}{(1-q)\sqrt{\log_2 x}}\bigg\} \bigg)^q.
\end{equation*}
In particular for $\theta = o(\frac{1}{\sqrt{\log_2 x}})$, we have 
$$
 \mathbb{E}\bigg[ \big|M(x,y;f)\big|\bigg] =o(\sqrt{y}).
$$
\end{theorem}
\begin{theorem}\label{theoreme_principal_2}
 Let $f$ be a Rademacher or Steinhaus random multiplicative function and $x$ be large. For $  \sqrt{ \log_2 x} \ll \log (\frac{x}{y}) $, we have uniformly for $ 0\leqslant q\leqslant 1$
 $$
 \mathbb{E}\bigg[ \big|M(x,y;f)\big|^{2q}\bigg] \asymp y^q.
$$
\end{theorem}
\begin{theorem}\label{theoreme_principal_3}
    Let $r$ be a large prime. Let $x,y$ be such that $x+y\leqslant r$ and \,\,\,$\log(\frac{x}{y})\ll ~\log\log x$, and let $\theta$ as in Theorem \ref{theoreme_principal_1}. We set $L:= \min\{ x+y,r/(x+y)\}+3$. We have for $L$ large enough and uniformly for $\theta > 0$ and $ 0\leqslant q \leqslant 1$
    \begin{equation}\label{inequality_theorem3}
        \frac{1}{r-1}\sum_{\chi \!\!\!\!\!\mod r}\big|M(x,y;\chi)\big|^{2q} \ll \bigg( y \min\bigg\{ 1, \theta \sqrt{\log_2 x} + \frac{1}{(1-q)\sqrt{\log_2 L}}\bigg\} \bigg)^q.
    \end{equation}
\end{theorem}
\begin{remark}
    Note that from Theorem \ref{theoreme_principal_1}, when $\theta = o(\frac{1}{\sqrt{\log_2 x}})$, we have 
    $$
    \mathbb{E}\bigg[ \big|M(x,y;f)\big|\bigg] =o(\sqrt{y}).
    $$
    In this domain, the quantity $ M(x,y;f)$ doesn't behave as gaussian. This extends Soundararajan and Xu's results \cite{Sound_Xu} of limiting distribution of 
    $$\frac{M(x,y;f)}{\sqrt{\mathbb{E}\big[ | M(x,y;f)|^{2}\big]}}$$ when $y$ is large comparing to $x$ (For more details, see \cite[Some corollaries 1.2.]{Harper}).
\end{remark}
\section{Sketch of the proof}
There are some technical distinctions between the Steinhaus and Rademacher cases. However the majority of their behavior and essential proof concepts align. Therefore, in this introductory discussion, we focus on the Steinhaus case. Theorem~ \ref{theoreme_principal_1} and Theorem~\ref{theoreme_principal_2} are established through the application of Harper's resilient approach in~\cite{Harper} to short intervals problem. The initial stage involves transforming the $L^1$ norm estimation into a specific average of squared random Euler products. Essentially, let $f$ be a Steinhaus multiplicative function and set $\delta:= \frac{x}{y}$. We prove that roughly speaking
\begin{equation}\label{price_bound}
     \mathbb{E}\bigg[ \bigg|\frac{1}{\sqrt{y}} M(x,y;f)\bigg|^{2q}\bigg] \approx \bigg(\frac{1}{\log x}\bigg)^q \mathbb{E}\bigg[\bigg(\frac{1}{\delta}\int_{-\delta}^{\delta}\big|F_{\sqrt{2x}}(it) \big|^2{\rm d}t\bigg)^q\bigg]
\end{equation}
where \begin{equation}\label{definition_F_z}
    F_{z}(s):= \prod_{p\leqslant z} \bigg(1-\frac{f(p)}{p^{1/2+s}}\bigg)^{-s}.
\end{equation} 
The difficult aspect lies in providing a sharp bound for right side hand of \eqref{price_bound}. Let's start with the upper bound. It's clear the straight application of H\"{o}lder's inequality doesn't give the required bound. In \cite{Harper}, the problem was reduced roughly speaking to $\mathbb{E}\bigg[\bigg(\int_{-1/2}^{1/2}\big|F_{\sqrt{2x}}(it) \big|^2{\rm d}t\bigg)^q\bigg] $. The importance difference between our problem \eqref{price_bound} and Harper's work is that the range of our integral in the right hand side of \eqref{price_bound} is larger. Indeed, we prove that as long as $\delta$ does not exceed a certain critical value, $\mathbb{E}\bigg[\bigg(\frac{1}{\delta}\int_{-\delta}^{\delta}\big|F_{\sqrt{2x}}(it) \big|^2{\rm d}t\bigg)^q\bigg]$  provides a ``good'' saving factor. However when $\delta$ becomes larger, we show that 
$$
\mathbb{E}\bigg[\bigg(\frac{1}{\delta\log x}\int_{-\delta}^{\delta}\big|F_{\sqrt{2x}}(it) \big|^2{\rm d}t\bigg)^q\bigg]\gg 1.
$$
In the prior studies (Harper \cite{Harper} and Xu \cite{Xu}), $\delta =\frac{1}{2}$. In the present investigation, $\delta$ takes large values, constituting a primary distinction between our work and the previous works.
Harper's approach initiates by introducing ``barrier events" that restrict the growth rate of every random partial Euler product across all values of $|t|\leqslant \delta$. In other words,  Harper define an event that control, for each $k$, the quantity $ \prod_{x^{{\rm e}^{-{k+1}}} < p\leqslant x^{{\rm e}^{-k}}} \big|1-\frac{f(p)}{p^{1/2+it}}\big|^{-1} $  by some ``expected growth". In addition to Harper's ``expected growth", in our case, we introduce an extra element arising from the extensive range of values within $|t|\leq \delta$: $\theta \sqrt{\log_2 x}$ where $\theta$ as defined in Theorem \ref{theoreme_principal_1}.\\
 Denote $\mathcal{G}$ to be this ``good" event, and $q':= \frac{1+q}{2}$. It is clear to see that
$$\begin{aligned} 
 &\mathbb{E} \bigg[\left(  \frac{1}{\delta\log x}   \int_{-\delta}^{\delta}|F_{\sqrt{2x}}(i t)|^2\right)^q\bigg]=\mathbb{E}\Bigg[\left(1_{\mathcal{G}} \frac{1}{\delta\log x} \int_{-\delta}^{\delta}|F_{\sqrt{2x}}(i t)|^2\right)^q\Bigg]
\\& \,\,\,\,\,+\mathbb{E}\Bigg[\left(1_{\mathcal{G} \text { fails }} \frac{1}{\delta\log x} \int_{-\delta}^{\delta}|F_{\sqrt{2x}}(i t)|^2\right)^q\Bigg] \\ & \leq \mathbb{E}\Bigg[\left(\mathbf{1}_{\mathcal{G}} \frac{1}{\delta\log x} \int_{-\delta}^{\delta}\left|F_{\sqrt{2x}}\left(i t\right)\right|^2\right)^q\Bigg]\\ &\,\,\,\,\,\,+\mathbb{P}\bigg[(\mathcal{G} \text { fails })^{\frac{q^{\prime}-q}{q^{\prime}}}\bigg]\left(\mathbb{E}\Bigg[\left(\frac{1}{\delta\log x} \int_{-\delta}^{\delta}\left|F_{\sqrt{2x}}\left(i t\right)\right|^2\right)^{q^{\prime}}\Bigg]\right)^{\frac{q}{q'}}.\end{aligned}$$\\
Here we applied H\"{o}lder's inequality in the second line. Following Harper \cite{Harper}, we prove first that the probability of ``$\mathcal{ G} $ fails" is ``small" and secondly the above expectation under the condition ``$\mathcal{G}$ holds" provides a saving factor: $\theta \log_2 x+ \frac{1}{(1-q)\sqrt{\log_2 x}}$. Following an iterative procedure, we get the required bound. Note that when $ \theta$ is larger than $1/\sqrt{\log_2 x}$, our bound is useless, since we can get a better bound by using a H\"{o}lder's inequality.  

For the lower bound, the strategy is relatively simple. we use the same approach as~\cite{Harper}: the problem is reduced to give an upper bound $ \mathbb{E}\left[\mathbf{1}_{\mathcal{G}} \frac{1}{\delta\log x} \int_{\mathcal{L}}\left|F_{\sqrt{2x}}\left(i t\right)\right|^2\right]^2$ for a chosen random set in $[-\delta, \delta]$ that verifies some ``good" barriers. We set $L(t)$ the event: those ``good" barriers occurs at point $t\in [-\delta ,\delta]$. The idea is that when we expend this expectation to $ \mathbb{E}\big[\mathbb{1}_{L(t_1)}\left|F_{\sqrt{2x}}\left(i t_1\right)\right|^2\mathbb{1}_{L(t_2)}\left|F_{\sqrt{2x}}\left(i t_2\right)\right|^2\big]$, we want this quality to behave as $\mathbb{E}\big[\mathbb{1}_{L(t_1)}\left|F_{\sqrt{2x}}\left(i t_1\right)\right|^2\big]\mathbb{E}\big[\mathbb{1}_{L(t_2)}\left|F_{\sqrt{2x}}\left(i t_2\right)\right|^2 \big]$.
\section{Preliminary results}
\subsection{Some known results}
We begin by stating some lemmas.

\begin{lemma}{(Parseval's identity).}\label{parseval} Let $(a_n)_{n \in \mathbb{N}}$ be sequence of complex numbers $A(s):= \sum_{n=1}^{+ \infty} \frac{a_n}{n^s}$ denotes the corresponding Dirichlet series and $\sigma_c$ denotes its abscissa of convergence. Then for any $\sigma > \max(0,\sigma_c)$, we have 
$$ \int_{0}^{+ \infty} \frac{\big|\sum_{n \leqslant   x} a_n \big|^2}{x^{1+2\sigma}}{\rm d}x=\frac{1}{2\pi}\int_{- \infty}^{+ \infty} \bigg| \frac{A(\sigma + it)}{\sigma +it} \bigg|^2 {\rm d}t.$$
\end{lemma}
\begin{proof}
See \cite[equation (5.26)]{MontgVaugh}.
\end{proof}

\noindent We define the following parameter
\begin{equation}\label{definition_a_f}
    a_f=
     \begin{cases}
        1 & \text{ if } f \text{ is a Rademacher multiplicative function} \\
        -1  & \text{ if } f \text{ is a Steinhaus multiplicative function}.
     \end{cases}
\end{equation}
\begin{lemma}{(Euler product result)}\label{lemmarachid2}
Let $f$ be a Rademacher or Steinhaus random multiplicative function.  For $t \in \mathbb{R}$ and $2 \leqslant   x \leqslant   y$, we have
$$ \mathbb{E} \Bigg[ \prod_{x < p \leqslant   y} \bigg| 1+a_f\frac{f(p)}{p^{1/2+it}} \bigg|^{2a_f} \Bigg] = \prod_{x < p \leqslant   y} \bigg( 1+\frac{a_f}{p} \bigg)^{a_f}. $$
\end{lemma}
\begin{proof}
See \cite[lemma 2.4]{Mastrostefano}.\goodbreak
\end{proof}

\begin{lemma}\label{lemma_Xu}
 Let $f$ be a Steinhaus random multiplicative function. Then for any \\$400<x \leqslant z$, $\sigma>-1 / \log z$ and $t \in \mathbb{R}$, we have
 \begin{equation}\label{eq_Xu1}
 \begin{aligned}
          & \!\!\!\!\!\!\!\!\!\!\!\!\!\!\!\!\!\!\!\!\!\!\!\!\mathbb{E}\left[\prod_{x<p \leqslant z}\left|1-\frac{f(p)}{p^{\frac{1}{2}+\sigma}}\right|^{-2}  \left| 1-\frac{f(p)}{p^{\frac{1}{2}+\sigma+i t}}\right|^{-2}\right]
          \\ & =\exp \left(\sum_{x<p \leqslant z} \frac{2+2 \cos (t \log p)}{p^{1+2 \sigma}}+O\left(\frac{1}{\sqrt{x} \log x}\right)\right).
 \end{aligned}
     \end{equation}
     Moreover, if $x>{\rm e}^{1 /|t|}$, then we further have \begin{equation}\label{eq_Xu2}
         \mathbb{E}\left[\prod_{x<p \leqslant z}\left|1-\frac{f(p)}{p^{\frac{1}{2}+\sigma}}\right|^{-2}\left|1-\frac{f(p)}{p^{\frac{1}{2}+\sigma+i t}}\right|^{-2}\right]=\exp \left(\sum_{x<p \leqslant z} \frac{2}{p^{1+2 \sigma}}+O(1)\right) .
     \end{equation} 
\end{lemma}
\begin{proof}
    The proof of \eqref{eq_Xu1} is established in the proof of \cite[lemma 6]{Harper}. To establish \eqref{eq_Xu2}, it is sufficient to proof $\sum_{x<p\leqslant z}\frac{\cos (t \log p)}{p^{1+2\sigma}} \ll 1$. This is a consequence of the prime number theorem (as exemplified, for instance, in the proof of \cite[lemma 5]{Harper}).
\end{proof}

An essential probabilistic outcome employed within Harper's approach is a variation of a well-known Gaussian random walk result. This result has ties to the ``ballot problem".
\begin{lemma}\label{gaussian_approximation}
Let $a\geqslant 1$. For any integer $n \geqslant 1$, let $G_1,...,G_n$ be an independent real Gaussian random variables, each having mean 0 and variance between $\frac{1}{20}$ and $20$. Then
$$
\mathbb{P}\bigg[  \sum_{m=1}^j G_m \leqslant  a+2 \log j + O(1) \text{ for all } j \leqslant  n  \bigg] \asymp \min\bigg\{1,\frac{a}{\sqrt{n}}\bigg\}.
$$
\end{lemma}
\begin{proof}
This is \cite[Probability Result 1]{Harper}.
\end{proof}
\begin{lemma}\label{perturabtion}
    Let $f$ be a Rademacher or Steinhaus random multiplicative function. Let $P>100$ and $\delta>1$. We set $K_P:= (\log P)^{1.01}$. For $P$ large enough, we have uniformly 
    $$
\begin{aligned}
    & \sum_{|k|\leqslant \delta K_P}   \frac{1}{\delta}\int_{-\frac{1}{2 K_P}}^{\frac{1}{2 K_P}}\mathbb{E}\bigg[  \bigg|F_P\bigg(\frac{ik}{ K_P}+it \bigg) - F_P\bigg(\frac{ik}{ K_P}\bigg)\bigg|^2  \bigg]  {\rm d}t
     \ll ( \log P)^{0.99} 
\end{aligned}
$$
where $F_P$ is defined in \eqref{definition_F_z}.
\end{lemma}
\begin{proof}
By using \cite[lemma 2]{Harper_charac}, we have
$$
\begin{aligned}
    & \sum_{|k|\leqslant \delta K_P}   \frac{1}{\delta}\int_{-\frac{1}{2 K_P}}^{\frac{1}{2 K_P}}\mathbb{E}\bigg[  \bigg|F_P\bigg(\frac{ik}{ K_P}+it \bigg) - F_P\bigg(\frac{ik}{ K_P}\bigg)\bigg|^2  \bigg] {\rm d}t
    \\ & = \sum_{|k|\leqslant \delta K_P}   \frac{1}{\delta}\int_{-\frac{1}{2 K_P}}^{\frac{1}{2 K_P}}\sum_{\substack{n \geqslant 1\\ P(n)\leqslant P }} \frac{|n^{-it}-1|^2}{n} {\rm d}t .
\end{aligned}
$$
We have 
$|n^{-it}-1|\ll \min\{|t|\log n,1 \} \leqslant \min \{ \frac{\log n}{ K_P},1 \}$. Thus 
$$
\sum_{\substack{n \leqslant P^{\log_2 P}\\ P(n)\leqslant P }} \frac{|n^{-it}-1|^2}{n} \ll \frac{(\log_2 P)^2}{(\log P)^{0.02}}\log P
$$
and for $n>P^{\log_2 P}$, we use the trivial bound $|n^{-it}-1|\ll 1$
$$
\sum_{\substack{n > P^{\log_2 P}\\ P(n)\leqslant P }} \frac{|n^{-it}-1|^2}{n} \ll \sum_{\substack{n > P^{\log_2 P}\\ P(n)\leqslant P }} \frac{1}{n}  \ll  \frac{\log P}{(\log P)^{0.01}}.
$$
This ends the proof.
\end{proof}
\subsection{Reduction of the proof}
    Note that it will suffice to prove Theorem \ref{theoreme_principal_1} and \ref{theoreme_principal_2} for $ 2/3\leqslant q \leqslant 1$. Indeed, following Harper \cite[Organization remarks 1.3]{Harper}, let $f$ be a Rademacher or Steinhaus Random multiplicative function, if we assume that 
    $$
        \mathbb{E}\bigg[ \bigg|\frac{1}{\sqrt{y}} M(x,y;f)\bigg|^{4/3}\bigg] \ll \bigg(  \min\bigg\{ 1, \theta \sqrt{\log_2 x} + \frac{3}{\sqrt{\log_2 x}}\bigg\} \bigg)^{2/3},
    $$
    then by H\"older's inequality, we have for $0\leqslant q \leqslant 2/3$
    $$
    \begin{aligned}
            \mathbb{E}\bigg[ \bigg|\frac{1}{\sqrt{y}} M(x,y;f)\bigg|^{2q} & \leqslant \bigg(\mathbb{E}\bigg[ \bigg|\frac{1}{\sqrt{y}}M(x,y;f)\bigg|^{4/3}\bigg]\bigg)^{3q/2}
            \\ & \ll  \bigg(  \min\bigg\{ 1, \theta \sqrt{\log_2 x} + \frac{1}{\sqrt{\log_2 x}}\bigg\} \bigg)^{q}.
    \end{aligned}
    $$
    For the lower bound, let $0 \leqslant q\leqslant 2/3$, if we assume that 
    $$
    \mathbb{E}\bigg[ \bigg|\frac{1}{\sqrt{y}} M(x,y;f)\bigg|^{4/3}\bigg]\asymp 1
    $$
    and
    $$
    \mathbb{E}\bigg[ \bigg|\frac{1}{\sqrt{y}}M(x,y;f)\bigg|^{3/2}\bigg]\asymp 1,
    $$
By applying, we have H\"older's, inequality
$$
\begin{aligned}
    1 \asymp & \mathbb{E}\bigg[ \bigg|\frac{1}{\sqrt{y}} M(x,y;f)\bigg|^{4/3}\bigg]  
    \\ & \leqslant \bigg(\mathbb{E}\bigg[ \bigg|\frac{1}{\sqrt{y}}M(x,y;f)\bigg|^{2q}\bigg]\bigg)^{\frac{1}{6(3/2-2q)}}\bigg(\mathbb{E}\bigg[ \bigg|\frac{1}{\sqrt{y}}M(x,y;f)\bigg|^{3/2}\bigg]\bigg)^{\frac{4-6q}{3(3/2-2q)}}
    \\& \ll \bigg(\mathbb{E}\bigg[ \bigg|\frac{1}{\sqrt{y}}M(x,y;f)\bigg|^{2q}\bigg]\bigg)^{\frac{1}{6(3/2-2q)}}.
\end{aligned} 
$$
Thus $$
\mathbb{E}\bigg[ \bigg|\frac{1}{\sqrt{y}} M(x,y;f)\bigg|^{2q}\bigg] \asymp 1.
$$    
\section{The case of random multiplicative functions}
\subsection{ The proof of the upper bound.}
\subsubsection{Intermediate upper bound.}\label{Intermediate upper bound}
The initial stage (Lemma \ref{sum_over_K}) involves establishing a connection between the $L^1$ norm of the random sums and a specific mean of the squared random Euler products.
For the sake of readability, we define the following quantities 
\begin{equation}\label{definition_G}
    G=G(x,\theta,q):= \min\bigg\{\theta \sqrt{\log_2 x}+\frac{1}{(1-q)\sqrt{\log_2 x}},1\bigg\},
\end{equation}
$$
x_{k}:=x^{{\rm e}^{-(k+1)}},
$$
and
$$
K:= \lfloor \log_3 x \rfloor.
$$
From this point onward, we consider $f$ as a Rademacher or Steinhaus random multiplicative function. $P(n)$ denotes the largest prime factor of $n$ and $P(1)=1$. We define also
\begin{equation}\label{def_M+_M-}
    M^{+}(x,y,z;f):= \sum_{\substack{x<n\leqslant x+y \\ P(n)\leqslant z }}f(n), \,\,\,\,\,\,\,\,\,\,\,\, M^{-}(x,y,z;f):= \sum_{\substack{x<n\leqslant x+y \\ P(n)> z }}f(n), 
\end{equation}
\begin{equation}\label{def_M_k_N_k}
M_{k}(x,y;f):= \sum_{\substack{ x< n \leqslant x+y\\ x_k< P(n) \leqslant x_{k-1}}} f(n), \,\,\,\,\,\,\,\,\, N_{k}(x,y;f):= \sum_{\substack{ x< n \leqslant x+y\\  P(n) \leqslant x_{k}}} f(n)   
\end{equation}
\begin{equation}\label{S_k}
    S_k(x,z;f):= \sum_{\substack{x<n\leqslant z \\ P(n)\leqslant x_{k} }}f(n).
\end{equation}
Note that
$$
M(x,y;f) = N_{K}(x,y;f)+\sum_{ 0\leqslant k \leqslant K} M_{k}(x,y;f) + M^{-}(x,y,x;f).
$$
We set $\lVert . \rVert_r:= (\mathbb{E}[|.|^r])^{\frac{1}{r}}$, by Minkowski’s inequality, for all $2/3 \leqslant q \leqslant 1$, we have
$$
\begin{aligned}
    \big\|M(x,y;f) \big\|_{2q}  \leqslant & \big\|   N_{K}(x,y;f)  \big\|_{2q} + \sum_{ 0\leqslant k \leqslant K} \big\|  M_{k}(x,y;f) \big\|_{2q}  +  \big\|  M^{-}(x,y,x;f)  \big\|_{2q}.
\end{aligned}
$$
Let's start by quantities $N_{K}(x,y;f)$ and $M^{-}(x,y,x;f).$
\begin{lemma}\label{p_n_less_x_power_-(K+1)} Let $f$ be a Rademacher or Steinhaus multiplicative function. Let $2 \leqslant y \leqslant x$ such that $ \frac{x}{y} \leqslant \exp \big(  \sqrt{\log_2 x} \big)$. Recall that $K= \lfloor \log_3 x \rfloor$. For $x$ large enough, we have uniformly for $2/3 \leqslant q \leqslant 1$ 
    $$
    \mathbb{E}\bigg[ \big| N_{K}(x,y;f) \big|^{2} \bigg]+ \,\,\, \mathbb{E}\bigg[ \big|   M^{-}(x,y,x;f)  \big|^{2} \bigg] \ll \frac{y}{\log x}.
    $$
\end{lemma}
\noindent This means that the both contributions are more than acceptable.
\begin{proof}
By H\"{o}lder inequality we have 
$$
\begin{aligned}
    \mathbb{E}\bigg[ \big| N_{K}(x,y;f) \big|^{2} \bigg]  \leqslant \Psi(x+y, x_K) -\Psi(x, x_K) 
\end{aligned}
$$
where 
$$
\Psi(x,z) := \sum_{\substack{ n\leqslant x \\ P(n) \leqslant z }}1.
$$
Since 
$ \frac{x}{y} \leqslant \exp \big(  \sqrt{\log_2 x} \big)$ and by applying \cite[theorem 5.1]{Hiledbrand_tenenbaum}, we get
$$
\Psi(x+y, x_K) -\Psi(x, x_K) \ll \frac{y}{(\log y)^{c \log_3 y}}\ll \frac{y}{\log x}.
$$
where $c$ is a constant.
In other hand, since $\log(\frac{x}{y})\ll \sqrt{\log_2 x} $ and by prime number theorem, we have
$$
\mathbb{E}\bigg[ \big|  M^{-}(x,y,x;f) \big|^2\bigg]= \mathbb{E}\bigg[ \bigg|\sum_{ x < p\leqslant x+y} f(p)\bigg|^2\bigg] \ll \frac{y}{\log x}. 
$$ 
\end{proof}
\noindent The problem is reduced to study
$
\big\| M_k(x,y;f) \big\|_{2q}.
$
\noindent We set 
\begin{equation}\label{definition_F_k}
    F_k(s):= \prod_{p\leqslant x_k} \bigg(1+a_f\frac{f(p)}{p^{1/2+s}}\bigg)^{a_f},
\end{equation}
where $a_f$ id defined by \eqref{definition_a_f}. We have the following useful following lemma:
\begin{lemma}[Harper]\label{lemmaharper01}
Let $f$ be a Rademacher or Steinhaus multiplicative function. Let $x$ be large. We have, uniformly for $\frac{2}{3} \leqslant  q \leqslant  1$ and $0\leqslant k \leqslant \lfloor \log_3  x \rfloor $
\begin{equation}\label{keyequaltion99}
    \max_{n \in \mathbb{Z}  }\frac{1}{(|n|+1)^{1/8}} \mathbb{E}\bigg[ \bigg( \int_{  -1/2}^{1/2} \big|F_k(it +in -k/\log x) \big|^2 {\rm d}t \bigg)^q\bigg] \ll \bigg(\frac{{\rm e}^{-k}\log x }{1+(1-q)\sqrt{\log_2  x}} \bigg)^q.
\end{equation}
\end{lemma}
\begin{proof}
    For Steinhaus case, it is a direct result from \cite[Key Proposition 1]{Harper} and \cite[Key Proposition 2]{Harper}. For Rademacher case, the result can be deduced easily from \cite[Key Proposition 3]{Harper} and \cite[Key Proposition 4]{Harper}. As it is done in the paragraph entitled: “Proof of the upper bound in Theorem 1, assuming Key Propositions 1 and 2” that works for both cases.
\end{proof}
\begin{lemma}\label{sum_over_K}
Let $f$ be a Rademacher or Steinhaus multiplicative function. Let $x$ be large and recall that $\delta=\frac{x}{y}$ and $K= \lfloor \log_3 x \rfloor$. Uniformly for all $2/3 \leqslant q \leqslant 1$, we have
$$
\begin{aligned}
    \sum_{ 0\leqslant k \leqslant K} \big\|  \frac{1}{\sqrt{y}}M_k(x,y;f) \big\|_{2q} \ll & \sum_{ 0\leqslant k \leqslant K}\bigg\|   \frac{1}{\delta \log x} \int_{-\delta}^{\delta}  \big|F_k(it -k/\log x) \big|^2{\rm d}t\bigg\|_{q}^{1/2} \\& +  \sum_{ 0\leqslant k \leqslant K} \bigg\| \frac{\delta}{\log x} \int_{\delta <|t| \leqslant \delta^5}\frac{\big|F_k(it -k/\log x) \big|^2}{\big|1/2 +it -k/\log x \big|^2} {\rm d}t \bigg\|_{q}^{1/2} 
   \\ &+ \bigg(\frac{1 }{1+(1-q)\sqrt{\log_2  x}} \bigg)^{1/2}.
\end{aligned}
$$
\end{lemma}
\begin{proof}
Let $\mathcal{F}_k$ be a $\sigma$-algebra generated by $\big\{ f(p), p \leqslant x_k \big\}$. We define $\mathbb{E}_k [\,.\,] := \mathbb{E}[\,. \,|\, \mathcal{F}_{k}]$. By H\"{o}lder's inequality, we have
$$
\begin{aligned}
    \big\| M_k(x,y;f) \big\|_{2q} & = \bigg(\mathbb{E}\bigg[\mathbb{E}_k \bigg[ \big| M_k(x,y;f) \big|^{2q}\bigg]\bigg]\bigg)^{1/2q}
    \\ & = \bigg(\mathbb{E}\bigg[\mathbb{E}_k \bigg[ \bigg| \sum_{d>1}^{x_k, x_{k-1}} f(d) N_k\bigg(\frac{x}{d},\frac{y}{d};f\bigg) \bigg|^{2q}\bigg]\bigg]\bigg)^{1/2q}
    \\ & \leqslant \bigg \| \sum_{d>1}^{x_k, x_{k-1}} \bigg|  N_k\bigg(\frac{x}{d},\frac{y}{d};f\bigg) \bigg|^2  \bigg\|_q^{1/2}.
\end{aligned}
$$
We proceed now by replacing $ \sum_{d>1}^{x_k, x_{k-1}}  \big| N_k\big(\frac{x}{d},\frac{y}{d};f\big) \big|^2$ by a
smoothed version. Set $X:= ~{\rm e}^{\sqrt{\log x}}$. We have
$$
\begin{aligned}
    \sum_{d>1}^{x_k, x_{k-1}} \bigg| N_k\bigg(\frac{x}{d},\frac{y}{d};f\bigg) \bigg|^2 & = \sum_{d>1}^{x_k, x_{k-1}} \frac{X}{d} \int_{d}^{d(1+\frac{1}{X})} \bigg|N_k\bigg(\frac{x}{d},\frac{y}{d};f\bigg) \bigg|^2 {\rm d}t
    \\ & \ll \sum_{d>1}^{x_k, x_{k-1}} \frac{X}{d} \int_{d}^{d(1+\frac{1}{X})} \bigg| N_k\bigg(\frac{x}{t},\frac{y}{t};f\bigg) \bigg|^2 {\rm d}t 
    \\ & \,\,\,\,\,\,\,\,\,\,\,\,\,\, + \sum_{d>1}^{x_k, x_{k-1}} \frac{X}{d} \int_{d}^{d(1+\frac{1}{X})} \bigg| S_k\bigg(\frac{x+y}{t},\frac{x+y}{d};f\bigg) \bigg|^2 {\rm d}t
    \\ & \,\,\,\,\,\,\,\,\,\,\,\,\,\, + \sum_{d>1}^{x_k, x_{k-1}} \frac{X}{d} \int_{d}^{d(1+\frac{1}{X})} \bigg| S_k\bigg(\frac{x}{t},\frac{x}{d};f\bigg) \bigg|^2 {\rm d}t.
\end{aligned}
$$
Using H\"{o}lder’s inequality and a mean square calculation, we have
\begin{equation}\label{inequality_0101}
    \begin{aligned}
    \mathbb{E}\bigg[ \bigg( \sum_{d \leqslant x+y}^{x_k, x_{k-1}} \frac{X}{d} \int_{d}^{d(1+\frac{1}{X})} \bigg| S_k\bigg(\frac{x+y}{t},\frac{x+y}{d};f\bigg) \bigg|^2 {\rm d}t \bigg)^q \bigg] & \leqslant \bigg( \sum_{d \leqslant x+y}^{x_k, x_{k-1}} \big(1+ \frac{x+y}{dX}\big)\bigg)^q.
\end{aligned}
\end{equation}
Using \cite[Number Theory Result 1]{Harper}, we have 
$$
 \sum_{d \leqslant x+y}^{x_k, x_{k-1}} 1 \ll 2^{-{\rm e}^{k}} \frac{x}{\log x}\ll 2^{-{\rm e}^{k}} \frac{y}{(\log x)^{0.98}}.
$$
We have also
$$
 \sum_{d \leqslant x+y}^{x_k, x_{k-1}} \frac{x+y}{mX} \ll \frac{x\log x}{X}.
$$
Recall that $X={\rm e}^{\sqrt{\log x}} $ and $\log(\frac{x}{y}) \ll \sqrt{\log_2 x}$. So taking 2$q$-th roots and summing over $0 \leqslant k \leqslant K$ leads to an acceptable overall contribution.
$$
    \sum_{0 \leqslant k \leqslant K} \Bigg( \mathbb{E}\bigg[ \bigg( \sum_{d \leqslant x+y}^{x_k, x_{k-1}} \frac{X}{d} \int_{d}^{d(1+\frac{1}{X})} \bigg| S_k\bigg(\frac{x+y}{t},\frac{x+y}{d};f\bigg) \bigg|^2 {\rm d}t \bigg)^q \bigg]\Bigg)^{1/2q} \ll \frac{\sqrt{y}}{(\log x)^{0.49}}.
$$
We prove, following the same steps, that
$$
   \sum_{0 \leqslant k \leqslant K} \Bigg( \mathbb{E}\bigg[ \bigg( \sum_{d \leqslant x+y}^{x_k, x_{k-1}} \frac{X}{d} \int_{d}^{d(1+\frac{1}{X})} \bigg| S_k\bigg(\frac{x}{t},\frac{x}{d};f\bigg) \bigg|^2 {\rm d}t \bigg)^q \bigg]\Bigg)^{1/2q}  \ll \frac{\sqrt{y}}{(\log x)^{0.49}}.
$$
We are left with 
\begin{equation}\label{eq_1}
        \mathbb{E}\bigg[ \bigg( \sum_{d \leqslant x+y}^{x_k, x_{k-1}} \frac{X}{d} \int_{d}^{d(1+\frac{1}{X})} \bigg| N_k\bigg(\frac{x}{t},\frac{y}{t};f\bigg) \bigg|^2 {\rm d}t \bigg)^q \bigg].
\end{equation}
 by swapping the order of the sum and integral, we get
 $$
 \begin{aligned}
      \sum_{d \leqslant x+y}^{x_k, x_{k-1}} \frac{X}{d} \int_{d}^{d(1+\frac{1}{X})} & \bigg| N_k\bigg(\frac{x}{t},\frac{y}{t};f\bigg) \bigg|^2 {\rm d}t 
      \\ & \ll \int_{x_k}^{(x+y)(1+\frac{1}{X})} \sum_{ \frac{t}{1+1/X}< d \leqslant t}^{x_k, x_{k-1}} \frac{X}{d} \bigg| N_k\bigg(\frac{x}{t},\frac{y}{t};f\bigg) \bigg|^2 {\rm d}t.
 \end{aligned}
 $$
 Using  a standard sieve estimate, we get the following bound
 \begin{equation}\label{Harper_sieve}
      \sum_{ \frac{t}{1+1/X}< d \leqslant t}^{x_k, x_{k-1}} \frac{X}{d} \ll \frac{1}{\log t}.
 \end{equation}
By making a substitution $z = x/t$ in the integral, we get
$$
\begin{aligned}
    &\!\!\!\!\!\!\!\!\!\!\!\!\!\!\!\!\!\!\!\!\!\!\!\!\!\int_{x_k}^{(x+y)(1+\frac{1}{X})} \sum_{ \frac{t}{1+1/X}< d \leqslant t}^{x_k, x_{k-1}} \frac{X}{d} \bigg| N_k\bigg(\frac{x}{t},\frac{y}{t};f\bigg) \bigg|^2 {\rm d}t 
    \\& \ll \int_{x_k}^{(x+y)(1+\frac{1}{X})} \frac{1}{\log t} \bigg| N_k\bigg(\frac{x}{t},\frac{y}{t};f\bigg) \bigg|^2 {\rm d}t
    \\ & = x \int^{ x/x_k}_{\frac{1}{(1+1/\delta)(1+1/X)}} \frac{1}{\log (x/z)} \big| N_k(z,z/\delta;f) \big|^2 \frac{{\rm d}z}{z^2}.
\end{aligned}
$$
To obtain a desirable relationship with the parameter $k$ in our final estimations, note that if $z \leqslant \sqrt{x}$, we have $\log(x/z)\gg \log x $, whereas if $  \sqrt{x} < z \leqslant x^{1-{\rm e}^{-(k+1)}}$ we have $\log(x/z) \gg {\rm e}^{-k}\log x$. Thus in any case we have $ \log(x/z) \gg z^{-2k/\log x} \log x$ so the above term is 
$$
\begin{aligned}
    & \ll  \frac{x}{\log x} \int^{ x/x_k}_{\frac{1}{(1+1/\delta)(1+1/X)}}  \big| N_k(z,z/\delta;f) \big|^2 \frac{{\rm d}z}{z^{2-2k/\log x}}
    \\ & \leqslant \frac{x}{\log x} \int^{ + \infty }_{0}  \big| N_k(z,z/\delta;f) \big|^2 \frac{{\rm d}z}{z^{2-2k/\log x}}.
\end{aligned}
$$
By applying Parseval's identity (Lemma \ref{parseval}) where $A(s)$ of Lemma \ref{parseval} in this case is 
$$
\begin{aligned}
    A(s)& = \int_{0}^{+\infty}N_k(z,z/\delta;f) z^{-s-1} {\rm d}z
    \\ &= \bigg(\bigg(\frac{1}{\delta}+1\bigg)^s -1\bigg)\int_{0}^{+\infty}\sum_{\substack{ n\leqslant z \\ P(n) \leqslant x_k}} f(n) z^{-s-1} {\rm d}z
    \\  &= \bigg(\bigg(\frac{1}{\delta}+1\bigg)^s -1\bigg)F_k(s),
\end{aligned}
$$
we get
\begin{equation}\label{eq_2}
\begin{aligned}
    &  \frac{x}{\log x} \int^{ + \infty }_{0}  \big| N_k(z,z/\delta;f) \big|^2 \frac{{\rm d}z}{z^{2-2k/\log x}}
 \\ &  \ll \frac{x}{\log x} \int_{-\infty}^{+ \infty} \frac{\big| (1+1/\delta)^{1/2+it - k/\log x} -1 \big|^2}{\big|1/2 +it -k/\log x \big|^2} \big|F_k(it -k/\log x) \big|^2{\rm d}t.  
\end{aligned}
\end{equation}
Note for $|t|\leqslant \delta$ we have 
$$
\frac{\big| (1+1/\delta)^{1/2+it - k/\log x} -1 \big|^2}{\big|1/2 +it -k/\log x \big|^2} \ll \frac{1}{\delta^2}
$$
and for $|t| > \delta $ we have
$$
\frac{\big| (1+1/\delta)^{1/2+it - k/\log x} -1 \big|^2}{\big|1/2 +it -k/\log x \big|^2} \ll \frac{1}{\big|1/2 +it -k/\log x \big|^2}.
$$
Thus, we have the term \eqref{eq_2} is 
$$
\ll A_1(x,\delta,k) + A_2(x,\delta,k) + A_3(x,\delta,k)
$$
where
$$
A_1(x,\delta,k):=    \frac{x}{\delta^2 \log x} \int_{-\delta}^{\delta}  \big|F_k(it -k/\log x) \big|^2 {\rm d}t= \frac{y}{\delta \log x} \int_{-\delta}^{\delta}  \big|F_k(it -k/\log x) \big|^2{\rm d}t,
$$
$$
A_2(x,\delta,k):=    \frac{x}{\log x} \int_{\delta <|t| \leqslant \delta^5}\frac{\big|F_k(it -k/\log x) \big|^2}{\big|1/2 +it -k/\log x \big|^2} {\rm d}t 
$$
and
$$
A_3(x,\delta,k):=  \frac{x}{\log x} \int_{|t| >\delta^5}\frac{\big|F_k(it -k/\log x) \big|^2}{\big|1/2 +it -k/\log x \big|^2} {\rm d}t.
$$
 Recall that $2q\geqslant 4/3>1$, thus by Minkowski's inequality, we get
 $$
 \begin{aligned}
      \sum_{ 0\leqslant k \leqslant K} \big\|  \frac{1}{\sqrt{y}}M_k(x,y;f) \big\|_{2q} \ll & \sum_{i=1}^3 \sum_{ 0\leqslant k \leqslant K}\big\|  A_i(x,\delta,k)\big\|_{q}^{1/2} + \frac{\sqrt{y}}{(\log x)^{0.49}}.
 \end{aligned}
 $$
 The expressions $A_1(x,\delta,k)$ and $A_2(x,\delta,k)$ are as required. Now, let's make a slight additional adjustment to $A_3(x,\delta,k)$. We have
$$
\begin{aligned}
    A_3(x,\delta,k) & \ll \frac{x}{\log x} \sum_{|n| > \delta^5  }\frac{1}{n^2}  \int_{ |t| \leqslant 1/2} \big|F_k(it +in -k/\log x) \big|^2 {\rm d}t
    \\ & \leqslant \frac{x}{ \delta^{5/4}\log x} \sum_{|n| > \delta^5  }\frac{1}{n^{7/4}}  \int_{ |t| \leqslant 1/2} \big|F_k(it +in -k/\log x) \big|^2 {\rm d}t 
    \\ & \leqslant \frac{y}{ \log x} \sum_{n \in \mathbb{Z}  }\frac{1}{(|n|+1)^{7/4}}  \int_{ |t| \leqslant 1/2} \big|F_k(it +in -k/\log x) \big|^2 {\rm d}t.
\end{aligned}
$$
We deduce the expectation 
$$
\begin{aligned}
    &\mathbb{E}\bigg[ \bigg( \frac{y}{ \log x} \sum_{n \in \mathbb{Z}  }\frac{1}{(|n|+1)^{7/4}}  \int_{ |t| \leqslant 1/2} \big|F_k( it +in -k/\log x) \big|^2 {\rm d}t \bigg)^q\bigg] 
    \\ & \leqslant \big(\frac{y}{ \log x}\big)^q \sum_{n \in \mathbb{Z}  }\frac{1}{(|n|+1)^{7q/4}} \mathbb{E}\bigg[ \bigg( \int_{ |t| \leqslant 1/2} \big|F_k(it +in -k/\log x) \big|^2 {\rm d}t \bigg)^q\bigg]
    \\ & \leqslant \big(\frac{y}{ \log x}\big)^q \max_{n \in \mathbb{Z}  }\frac{1}{(|n|+1)^{1/8}} \mathbb{E}\bigg[ \bigg( \int_{ |t| \leqslant 1/2} \big|F_k(it +in -k/\log x) \big|^2 {\rm d}t \bigg)^q\bigg].
\end{aligned}
$$
By Lemma \ref{lemmaharper01}, we deduce then 
\begin{equation}
    \mathbb{E}\big[ \big( A_3(x,\delta,k) \big)^q\big] \ll \bigg(\frac{y{\rm e }^{-k} }{1+(1-q)\sqrt{\log_2  x}} \bigg)^q.
\end{equation}
This ends the proof.
\end{proof}
\subsubsection{Multiplicative Chaos}

This section is dedicated to prove the Proposition \ref{prop_harper01}. This Proposition is an adjusted version of \cite[Key proposition 1]{Harper}, \cite[Key proposition 2]{Harper} for Steinhaus case and \cite[Key proposition 3]{Harper}, \cite[Key proposition 4]{Harper} for Rademacher case. We will keep many of the constants analogous to the ones that arise in Harper's proof \cite{Harper}.
   \begin{prop}\label{prop_harper01}
                Let $f$ be a Rademacher or Steinhaus random multiplicative function and $x$ be large. Let $\theta >0$ be such that $\frac{x}{y}=(\log x)^{\theta}$. For $0\leqslant k \leqslant K= \lfloor \log_3 x \rfloor$ and $2/3\leqslant q\leqslant 1$, we have uniformly
        \begin{equation}\label{eq_chao1}
        \begin{aligned}
                \!\!\!\!\!\mathbb{E}\bigg[\bigg(  \frac{{\rm e}^k}{\delta\log x }\int_{-\delta}^{\delta} &\big| F_{k}(-k/\log x +it)\big|^2 {\rm d}t\bigg)^q   \bigg]
                \\ & \ll \bigg(\min\bigg\{ 1, \theta \sqrt{\log_2 x} +\frac{1}{(1-q)\sqrt{\log_2 x}}\bigg\}\bigg)^q
        \end{aligned}
        \end{equation}
        and 
        \begin{equation}\label{eq_chao2}
        \begin{aligned}
            \mathbb{E}\bigg[\bigg( \frac{{\rm e}^k\delta }{\log x }\int_{ \delta< |t|\leqslant \delta^5}&\frac{\big| F_{k}(  - k/\log x +it)\big|^2}{\big| 1/2-k/\log x +it \big|^2} {\rm d}t\bigg)^q   \bigg] \\ & \ll \bigg( \min\bigg\{ 1, \theta \sqrt{\log_2 x} +\frac{1}{(1-q)\sqrt{\log_2 x}}\bigg\}\bigg)^q.
        \end{aligned}
        \end{equation}
    \end{prop}
    \begin{proof}[Proof of Theorem \ref{theoreme_principal_1} assuming Proposition \ref{prop_harper01}.]
        It is a direct result form\\ Lemma~\ref{sum_over_K} and Proposition \ref{prop_harper01}.
    \end{proof}
\noindent Let $B$ a large constant, we define
\begin{equation}\label{def_D(t)}
    D(t)=\left\{
  \begin{array}{@{}ll@{}}
    \lceil \log (1/|t|)\rceil+ B+1  & \text{if}\ \frac{1}{\sqrt{\log_2 x}}< |t|\leqslant  \frac{1}{2}  \\
    B+1 & \text{ if } 1/2<|t|\leqslant \delta^5.
  \end{array}\right.
\end{equation}
Let $I_{r}$ to be the $r$-th increment of $F(s)$ Euler product 
\begin{equation}\label{def_I_r}
    I_{r}(u,v)=\prod_{x_{r+1} < p \leqslant  x_r }\bigg(1+a_f\frac{f(p)}{p^{1/2+iv+u}} \bigg)^{a_f}.
\end{equation}
In order to simplify the notations, we set 
\begin{equation}\label{def_u_m}
    u_m = u_m(t):= \lfloor\log_2  x \rfloor-D(t)-m.
\end{equation}
We set 
\begin{equation}\label{def_of_S_R}
    R=R(t):=\lfloor\log_2  x \rfloor-D(t), \,\,\, S:= \lfloor\log_2  x \rfloor-B
\end{equation}
and $x_{r}:= x^{{\rm e}^{-(r-1)}}$. We observe that $R(t) \leqslant S$ for any $\frac{1}{\sqrt{\log_2 x}} \leqslant |t| \leqslant  \delta^5$. To prove 
$$\mathbb{E}\bigg[\bigg|\sum_{n \leqslant x}f(n)\bigg|\bigg] \ll \frac{\sqrt{x}}{(\log_2 x)^{1/4}}, $$
Harper \cite{Harper} has introduced an event $``\mathcal{G}^{Rad}(k)"$ for Rademacher case and $``\mathcal{G}(k)"$ for Steinhaus case (see the introduction \cite[section 4.1]{Harper} and \cite[section 4.4]{Harper}). We shall introduce an analogue of these events, which is slightly different from the Harper's original definitions. In order to define it, we need first to discretize $t$. As Harper \cite[section 4.1]{Harper}, for each $|t|\leqslant  \delta^5$, set $w(-1):=t$, and iteratively for each $0 \leqslant  r \leqslant  \log_2  x -2$
\begin{equation}
    w(r)=\frac{\lfloor w(r-1) \log x_r \log_2 x_r  \rfloor}{\log x_r \log_2 x_r}.
\end{equation}
We denote
\begin{equation}\label{def_v_m}
    v_m(t)=v_m:=w(u_m).
\end{equation} 
Now we define the event $\mathcal{G}_q(k,t)$ to be
\begin{equation}
     - a_j(q,t) \leqslant  \sum^{j}_{m=1}  \log \left|I_{u_m}\left(\frac{k}{\log x}, v_m\right)\right| \leqslant  a_j(q,t)
\end{equation}
for all $k\leqslant  j \leqslant  R(t)$, where
\begin{equation}\label{definition_a_k}
    a_j(q,t):= j+a(q,t)+2 \log j
\end{equation}
and
\begin{equation}\label{deinition_of_c}
    a(q,t):= \frac{C}{1-q} +5\theta \log_2 x  + D(t) + \log (D(t)+1)
\end{equation}
for large constant $C>B>1$. Furthermore, we let $\mathcal{G}_q(k)$ denote the event that $\mathcal{G}_q(k,t)$ occurs for all $1/\sqrt{\log_2  x}<|t|\leqslant   \delta^5$. This event allows us to control the Euler product in the ``logarithmic scale''. We begin by giving an estimate of the probability when the event $\mathcal{G}_q(k)$ fails.
\begin{lemma}\label{key4}

Let $f$ be a Rademacher or Steinhaus multiplicative function. For $x$ large enough, and uniformly for $\frac{2}{3} \leqslant  q \leqslant  1-\frac{1}{\sqrt{\log_2 x}}$ and $0\leqslant k\leqslant K$, we have
$$
\mathbb{P}\big[ \mathcal{G}_q(k) \text{ Fails} \big] \ll { \rm e } ^{\frac{-2C}{1-q} }
$$
where $C$ is the same constant as defined in \eqref{deinition_of_c}.
\end{lemma}
\begin{proof}
Following \cite[ Key propositions 2]{Harper} and \cite[ Key propositions 4]{Harper}, we have
\begin{equation}\label{eqexplaindiscrti}
    \begin{aligned}
\mathbb{P}\big[ \mathcal{G}_q(k) \text{ Fails} \big] \leqslant \Sigma_1 +\Sigma_2
\end{aligned}
\end{equation}
where
$$
\Sigma_1 :=   \sum_{k \leqslant j \leqslant  S} \mathbb{P}\Bigg[ \prod_{m=1}^{j}\left|I_{u_m}\left(\frac{k}{\log x}, v_m(t)\right)\right| >{\rm e}^{a_j(q,t)} \text{ for some } 1/\sqrt{\log_2  x}<|t| \leqslant  \delta^5\Bigg]
$$
and
$$
\Sigma_2 := \sum_{k \leqslant j \leqslant  S} \mathbb{P}\Bigg[ \prod_{m=1}^{j}\left|I_{u_m}\left(\frac{k}{\log x}, v_m(t)\right)\right|^{-1} >{\rm e}^{a_j(q,t)} \text{ for some } 1/\sqrt{\log_2  x}<|t| \leqslant  \delta^5\Bigg].
$$
where $v_m(t)$ is defined in \eqref{def_v_m} and \eqref{def_of_S_R} respectively. Since for each ${\rm e}^{-r-1} \leqslant |t|<{\rm e}^{-r}$, $u_m$ and $v_m(t)$ depend only on $r$. Recall $u_m$ is defined in \eqref{def_u_m}, note that $u_m(t)=~u_m({\rm e}^{-r-1})$ and $v_m(t)=~v_m({\rm e}^{-r-1})$ take only a discrete points in $ \mathcal{T}(x,j)$ with
$$
\begin{aligned}
\mathcal{T}(x,j):=\Bigg\{ & \frac{n}{{\rm e}^{j+B+\langle \log_2 x \rangle} (j+B+\langle \log_2 x \rangle) }: 
\\& \,\,\,\,\,\,\,\,\,\,\,\,\,\,\,\,\,\,\,\,\,\,\,\,\,\,\,\,\,|n| \leqslant   \delta^5{\rm e}^{j+B+\langle \log_2 x \rangle}  (j+B+\langle \log_2 x \rangle) \Bigg\},
\end{aligned}
$$
where $\langle \log_2 x \rangle:=  \log_2 x - \lfloor \log_2 x \rfloor$. Let $\Tilde{a}_j(q) := j+2\log j +\frac{C}{1-q}+\theta\log_2 x$. It's clear that for every $\frac{1}{\sqrt{\log_2 x}}\leqslant |t|\leqslant \delta^5$ and $1\leqslant j\leqslant S$, we have $a_j(q,t) \geqslant \Tilde{a}_j(q) $. We have
$$
\Sigma_1 \ll \!\!\! \sum_{k \leqslant  j \leqslant  S} \sum_{b \in \mathcal{T}(x,j)}\mathbb{E}\left[\prod_{m=1}^{j}\left|I_{u_m}\left(\frac{k}{\log x}, b\right)\right|>{\rm e}^{\Tilde{a}_j(q)} \right].
$$
By Chebychev's inequality, and by using $\prod_{m=1}^{j}\mathbb{E}\left[ \left|I_{u_m}\left(\frac{k}{\log x}, b\right)\right|^2 \right] \ll {\rm e}^j$ (which is combination of Lemma \ref{lemmarachid2} and Mertens' theorem), we have $\Sigma_1$ is less than
$$
\begin{aligned}
& \!\!\!\!\!\!\!\!\!\!\!\!\!\!\!\!\!\!\!\!\!\!\!\!\!\!\!\!\!\!\sum_{k \leqslant  j \leqslant  S}   \delta(j+B+1){\rm e}^{j+B+1} {\rm e}^{-2\Tilde{a}_j(q)} {\rm e}^{j}
\\ & \ll { \rm e } ^{\frac{-2C}{1-q} }\sum_{k \leqslant  j \leqslant  S}  \delta (j+B+1) {\rm e}^{r-10 \theta \log_2 x -4 \log j}   
\\ & =  { \rm e } ^{\frac{-2C}{1-q} } \sum_{k \leqslant  j \leqslant  S}   \frac{j+B+1}{j^4}  
\\ & \ll { \rm e } ^{\frac{-2C}{1-q}}.
\end{aligned}
$$
Since $ \mathbb{E}\left[ \left|I_{u_m}\left(\frac{k}{\log x}, v_m({\rm e }^{-r-1})\right)\right|^2 \right]$ has the same estimate as $ \mathbb{E}\left[ \left|I_{u_m}\left(\frac{k}{\log x}, v_m({\rm e }^{-r-1})\right)\right|^{-2} \right]$ and following the same steps as above, we get $\Sigma_2 \ll { \rm e } ^{\frac{-2C}{1-q}}$. This ends the proof.
\end{proof}
\begin{remark}
The reason behind discretizing $t$ to $v_m(t)$ is to deal with the computation of the probability on the right hand side of the inequality \eqref{eqexplaindiscrti}. Instead of working with a continuous $t$, we have replaced it by a sequence of approximations, this allows us to compute easily this probability.
\end{remark}

\noindent Given a probability space $(\mathbb{P},\Omega,\mathcal{F})$, we can introduce other probability measures and obtain other probability spaces with the same sample space $\Omega$ and the same $\sigma$-algebra of events $\mathcal{F}$. One “easy” way to do that is to start with a nonnegative random variable $\Theta$ such that $\mathbb{E}[\Theta]=1$. Then we can obtain another probability measure 
$$ \widetilde{\mathbb{P}}[A]:= \mathbb{E}[\mathbb{1}_A \Theta] \text{ for all } A \in \mathcal{F}.$$
Let $\widetilde{\mathbb{E}}$ denote the expectation associated to  $\widetilde{\mathbb{P}}$ (i.e for any $Y$ random variable $\widetilde{\mathbb{E}}[Y]:= \mathbb{E}[\Theta Y] $).
With the same way, we define the following probability.\\
$$
\widetilde{\mathbb{P}}_t(A):=\frac{\mathbb{E} \bigg[  \mathbb{1}_A \prod_{p \leqslant  x^{{\rm e}^{-1}}}\left|1+a_f\frac{f(p)}{p^{1 / 2+it}}\right|^{2a_f}\bigg]}{\mathbb{E}\bigg[ \prod_{p \leqslant  x^{{\rm e}^{-1}}}\left|1+a_f\frac{f(p)}{p^{1 / 2+it}}\right|^{2a_f}\bigg]},
$$
where $A$ is an event from $\mathcal{F}$. The next lemma provides an upper bound of $\widetilde{\mathbb{P}}_t\big[\mathcal{G}_q(k,t)\big]$ uniformly for all $\frac{2}{3} \leqslant  q \leqslant  1$, $\frac{1}{\log_2 x} \leqslant |t|\leqslant  \delta^5$ and $0\leqslant k \leqslant K$.
\begin{lemma}\label{key3}
Let $f$ be a Rademacher or Steinhaus multiplicative function. We have uniformly for $\frac{2}{3} \leqslant  q \leqslant  1$ and $\frac{1}{\log_2 x} \leqslant |t|\leqslant  \delta^5$
$$
\widetilde{\mathbb{P}}_t\big[\mathcal{G}_q(k,t)\big] \ll \frac{C  }{(1-q)\sqrt{\log_2  x}}+\theta\sqrt{\log_2 x}+\frac{ D(t) }{\sqrt{\log_2  x}},
$$
where $C$ is the same constant as defined in  \eqref{deinition_of_c}.
\end{lemma}

\begin{proof}[Proof of Proposition \ref{prop_harper01} assuming Lemma \ref{key3}.]
The proof is based on the same argument as in the Steinhaus case \cite[section 4.1]{Harper} in the paragraph entitled ``Proof of the upper bound in Theorem 1, assuming Key Propositions 1 and 2''. By following Harper \cite[section 4.4]{Harper}, we have immediately by H\"{o}lder's inequality and Lemma~\ref{lemmarachid2}
$$
\begin{aligned}
\mathbb{E}\bigg[ \bigg(\int_{\substack{  |t| \leqslant 1/ \sqrt{\log_2  x}}} & \big|F_k(-k/\log x+it)\big|^2 {\rm d}t \bigg)^{q}\bigg] 
\\ &  \leqslant  \bigg(\int_{\substack{ |t| \leqslant 1/ \sqrt{\log_2  x}}} \mathbb{E}\bigg[ \big|F_k(-k/\log x+it)\big|^2 \bigg] {\rm d}t\bigg)^{q}
\\ & \ll \bigg( \frac{\log x}{{\rm e}^{k}\sqrt{\log_2  x}}\bigg)^{q}.
\end{aligned}
$$
It suffices to prove 
        \begin{equation}\label{eq_001}
            \mathbb{E}\bigg[\bigg(  \frac{{\rm e}^k}{G\delta \log x }\int_{\substack{|t|\leqslant \delta \\ |t|\geqslant \frac{1}{\sqrt{\log_2 x}}}} \big| F_{k}(-k/\log x +it)\big|^2 {\rm d}t\bigg)^q   \bigg] \ll 1
        \end{equation}
        and
        \begin{equation}\label{eq_002}
            \mathbb{E}\bigg[\bigg(  \frac{{\rm e}^k\delta }{G \log x }\int_{ \delta< |t|\leqslant \delta^5}\frac{\big| F_{k}(-k/\log x +it)\big|^2}{\big| 1/2-k/\log x +it \big|^2} {\rm d}t\bigg)^q   \bigg] \ll 1
        \end{equation}
        where $G$ is defined in \eqref{definition_G}.
        We prove only the first inequality \eqref{eq_001}, the second one \eqref{eq_002} is proven following exactly the same steps.
        Following Harper \cite{Harper}, more precisely in \textit{Proof of the upper bound in Theorem 1, assuming Key Propositions 1 and 2}.
        \noindent If $1-\frac{1}{\sqrt{\log_2 x}} \leqslant  q \leqslant  1$, we have $G\asymp 1$, thus the upper bound we need to prove is 
$$
\mathbb{E}\bigg[\bigg(  \frac{{\rm e}^k}{\delta \log x }\int_{\substack{|t|\leqslant \delta \\ |t|\geqslant \frac{1}{\sqrt{\log_2 x}}}} \big| F_{k}(-k/\log x +it)\big|^2 {\rm d}t\bigg)^q   \bigg] \ll 1.
$$
By H\"{o}lder's inequality
$$
\begin{aligned}
 \mathbb{E}\bigg[\bigg( \int_{\substack{|t|\leqslant \delta \\ |t|\geqslant \frac{1}{\sqrt{\log_2 x}}}} \big| F_{k}(-k/\log x +it)\big|^2 {\rm d}t\bigg)^q   \bigg] & \leqslant \bigg(  \int_{\substack{|t|\leqslant \delta \\ |t|\geqslant \frac{1}{\sqrt{\log_2 x}}}} \mathbb{E}\big[\big| F_{k}(-k/\log x +it)\big|^2\big] {\rm d}t\bigg)^q.
 \end{aligned}
$$
By using Lemma~\ref{lemmarachid2} together with Mertens' theorem, we have $\mathbb{E}\big[\big|F_k(-k/\log x+it)\big|^2 \big]\ll \frac{\log x}{{\rm e}^k}$. We get then
$$
 \frac{{\rm e}^k}{\delta \log x}  \int_{\substack{|t|\leqslant \delta \\ |t|\geqslant \frac{1}{\sqrt{\log_2 x}}}} \mathbb{E}\big[\big| F_{k}(-k/\log x +it)\big|^2\big] {\rm d}t\ll 1.
$$
Then, it suffices to prove
$$
            \mathbb{E}\bigg[\bigg(  \frac{{\rm e}^k}{G\delta \log x }\int_{\substack{|t|\leqslant \delta \\ |t|\geqslant \frac{1}{\sqrt{\log_2 x}}}} \big| F_{k}(-k/\log x +it)\big|^2 {\rm d}t\bigg)^q   \bigg] \ll 1
$$
uniformly for all $\frac{2}{3} \leqslant  q \leqslant  1-\frac{1}{\sqrt{\log_2 x}} $.  
        Let $\frac{1}{\sqrt{\log_2 x}}\leqslant s\leqslant \frac{1}{6} $ we define $R(s)$ as
        $$
        R(s) := \sup_{1-2s \leqslant q \leqslant 1-s} \mathbb{E}\bigg[\bigg(  \frac{{\rm e}^k}{G\delta \log x }\int_{\substack{|t|\leqslant \delta \\ |t|\geqslant \frac{1}{\sqrt{\log_2 x}}}} \big| F_{k}(-k/\log x +it)\big|^2 {\rm d}t\bigg)^q   \bigg].
        $$
        Note that
        $$
        \begin{aligned}
        &\!\!\!\!\!\!\!\!\!\!\!\!\!\!\!\!\!\!\!\!\!\!\!\!\!\!\!\mathbb{E}\bigg[\bigg( \int_{\substack{|t|\leqslant \delta \\ |t|\geqslant \frac{1}{\sqrt{\log_2 x}}}} \big| F_{k}(-k/\log x +it)\big|^2 {\rm d}t\bigg)^q \bigg]
        \\ & \leqslant \mathbb{E}\bigg[\bigg( \int_{\substack{|t|\leqslant \delta \\ |t|\geqslant \frac{1}{\sqrt{\log_2 x}}}} \mathbb{1}_{\mathcal{G}_q(k,t)} \big| F_{k}(-k/\log x +it)\big|^2 {\rm d}t\bigg)^q \bigg]
        \\ & \,\,\,\,\,\,\,+ \mathbb{E}\bigg[\bigg( \mathbb{1}_{\mathcal{G}_q(k) \text{ Fails}} \int_{\substack{|t|\leqslant \delta \\ |t|\geqslant \frac{1}{\sqrt{\log_2 x}}}} \big| F_{k}(-k/\log x +it)\big|^2 {\rm d}t\bigg)^q \bigg].
        \end{aligned}
        $$
        \noindent Now by using $\mathbb{E}\big[ \big|F_k(-k/\log x+it)\big|^2\big] \asymp \frac{\log x}{{\rm e}^k} $ uniformly for all $t$, following by Lemma \ref{key3}, we have then the last term above is
    $$
    \begin{aligned}
     & \ll\bigg(\frac{1}{\delta G}\int_{\substack{|t| \leqslant  \delta \\ |t| > 1/\sqrt{\log_2  x}}} \widetilde{\mathbb{P}}_t\big[\mathbb{1}_{\mathcal{G}_q(k,t)} \big] {\rm d}t \bigg)^{q}
     \ll \bigg( C + \frac{1}{\delta G\sqrt{\log_2 x}} \int_{\substack{|t| \leqslant  \delta \\ |t| > 1/\sqrt{\log_2  x}}} D(t) {\rm d}t  \bigg)^{q}.
    \end{aligned}
    $$
    Finally, by observing that $\frac{1}{\delta}\int_{|t|\leqslant  \delta}D(t) {\rm d}t \leqslant  2B$, we get
            $$
        \begin{aligned}
         \mathbb{E}\bigg[\bigg( \mathbb{1}_{\mathcal{G}_q(k)} \frac{{\rm e}^k}{G\delta \log x }\int_{\substack{|t|\leqslant \delta \\ |t|\geqslant \frac{1}{\sqrt{\log_2 x}}}} \big| F_{k}(-k/\log x +it)\big|^2 {\rm d}t\bigg)^q   \bigg]
            \ll C^q\leqslant C.
        \end{aligned}
        $$
        Next, for each $1-2s \leqslant q \leqslant 1-s$ we set $q^{'} = \frac{1+q}{2}$, so that $1-s \leqslant q^{'} \leqslant 1-s/2$. Then by H\"{o}lder’s inequality with exponents $\frac{q^{'}}{q^{'}-q}$ and $\frac{q^{'}}{q}$, we have
        $$
        \begin{aligned}
        & \,\,\,\,\, \mathbb{E}\bigg[\bigg( \mathbb{1}_{\mathcal{G}_q(k) \,\, \text{fails}} \frac{{\rm e}^k}{G\delta \log x }\int_{\substack{|t|\leqslant \delta \\ |t|\geqslant \frac{1}{\sqrt{\log_2 x}}}} \big| F_{k}(-k/\log x +it)\big|^2 {\rm d}t\bigg)^q   \bigg]
        \\ & \leqslant \bigg(\mathbb{E}\big[\mathbb{1}_{\mathcal{G}_q(k) \,\, \text{fails}} \big]\bigg)^{(q^{'}-q)/q^{'}} \bigg(\mathbb{E}\bigg[\bigg(  \frac{{\rm e}^k}{G\delta \log x }\int_{\substack{|t|\leqslant \delta \\ |t|\geqslant \frac{1}{\sqrt{\log_2 x}}}} \big| F_{k}(-k/\log x +it)\big|^2 {\rm d}t\bigg)^{q^{'}}   \bigg]\bigg)^{q/q^{'}}
        \\ & \leqslant \big(\mathbb{P}\big[\mathcal{G}_q(k) \,\, \text{ fails} \big]\big)^{s/2} \bigg(\mathbb{E}\bigg[\bigg(  \frac{{\rm e}^k}{G\delta \log x }\int_{\substack{|t|\leqslant \delta \\ |t|\geqslant \frac{1}{\sqrt{\log_2 x}}}}\big| F_{k}(-k/\log x +it)\big|^2 {\rm d}t\bigg)^{q^{'}}   \bigg]\bigg)^{q/q^{'}}.
        \end{aligned}
        $$
        By Lemma \ref{key4}, we have $\big(\mathbb{P}\big[\mathcal{G}_q(k) \,\, \text{ fails} \big]\big)^{s/2} \ll {\rm e}^{-\frac{-Cs}{1-q}} \leqslant {\rm e}^{ -C/2 }$. Let $N$ be large enough such that $\frac{1}{\sqrt{\log x}}< \frac{s}{2^{N}} \leqslant \frac{2}{\sqrt{\log_2 x}}$, we call this quantity $J_s:= \frac{s}{2^N}$. Since $\frac{q}{q^{'}} \leqslant 1$, we have then
        $$
        R(s) \ll C+{\rm e}^{-C/2} (1+R(s/2))\ll C+{\rm e}^{-C/2}R(s/2).
        $$
        Let $A>0$ such that $R(s)\leqslant AC+A{\rm e}^{-C/2}R(s/2)$. We choose $C$ large enough such that $A{\rm e}^{-C/2} <1$.
        By iterating and replacing $s$ by $s/2,s/4,...$ we get at the end
        $$
        \begin{aligned}
            R(s) & \leqslant AC \bigg( \sum_{0\leqslant j\leqslant N-1} \big(A{\rm e}^{-C/2} \big)^{j} \bigg) +\big(A{\rm e}^{-C/2} \big)^{N}R(J_s)
            \\ & \ll 1+ R(J_s).
        \end{aligned}
        $$
        By using H\"{o}lder’s inequality and $\mathbb{E}\big[\big| F_{k}(-k/\log x +it)\big|^2 \asymp \frac{\log x}{{\rm e}^{k}}$ uniformly on $t$, we get
        $$
        \begin{aligned}
            R(J_s) & \ll \sup_{1-2J_s < q \leqslant 1-J_s} \mathbb{E}\bigg[\bigg(  \frac{{\rm e}^k}{G\delta \log x }\int_{\substack{|t|\leqslant \delta \\ |t|\geqslant \frac{1}{\sqrt{\log_2 x}}}} \big| F_{k}(-k/\log x +it)\big|^2 {\rm d}t\bigg)^q   \bigg]
            \\ & \leqslant \sup_{1-2J_s < q \leqslant 1-J_s} \bigg(  \frac{{\rm e}^k}{G\delta \log x }\int_{\substack{|t|\leqslant \delta \\ |t|\geqslant \frac{1}{\sqrt{\log_2 x}}}} \mathbb{E}\big[\big| F_{k}(-k/\log x +it)\big|^2\big] {\rm d}t\bigg)^q
            \\ & \ll \sup_{1-2J_s < q \leqslant 1-J_s} \bigg(\frac{1}{G}\bigg)^q.
        \end{aligned}
        $$ 
        Note that for $1-2J_s < q \leqslant 1-J_s$ we have $\frac{1}{1-q}\asymp \frac{1}{\sqrt{\log_2 x}}$. This implies that $G = \theta \sqrt{\log_2 x} + \frac{1}{(1-q)\sqrt{\log_2 x}}\gg 1$. Thus we have $R(s)\ll 1$ uniformly for $\frac{1}{\sqrt{\log_2 x}}\leqslant s\leqslant \frac{1}{6} $. This ends the proof.
\end{proof}
\noindent\textbf{Proof of Lemma \ref{key3}.}\\
 We start by the following lemma.
\begin{lemma}\label{Key3_1}
Let $f$ be a Rademacher or Steinhaus multiplicative function. We have uniformly for $\frac{2}{3} \leqslant  q \leqslant  1$, $1/\sqrt{\log_2 x} \leqslant  |t|\leqslant  \delta^5$ and $ 0 \leqslant k \leqslant K$
\begin{equation}\label{upperbound_event_with_gaussian}
\widetilde{\mathbb{P}}_t[\mathcal{G}_q(k,t)] \ll \mathbb{P}\bigg[  \sum_{m=1}^j G_m \leqslant  a(q,t)+2 \log j+O(1) \text{ for all }  k\leqslant  j \leqslant  R  \bigg]
\end{equation}
and 
\begin{equation}\label{lower_bound_for_part_2}
    \widetilde{\mathbb{P}}_t[\mathcal{G}_q(k,t)] \gg \mathbb{P}\bigg[ -2k -O(1) \leqslant \sum_{m=1}^j G_m \leqslant  a(q,t)+2 \log j+O(1) \text{ for all } k\leqslant j \leqslant  R  \bigg]
\end{equation}
where $G_m$ are independent Gaussian random variables, each having mean $0$ and variance: for Rademacher case $\sum_{x_{u_m}^{1/{\rm e}} < p \leqslant  x_{u_m}} \frac{1+\cos(2v_m(t)\log p)}{2p}$, and for Steinhaus case $\sum_{x_{u_m}^{1/{\rm e}} < p \leqslant  x_{u_m}} \frac{1}{2p}.$

\noindent In particular, we have 
\begin{equation}\label{equation_P_asymp}
    \widetilde{\mathbb{P}}_t[\mathcal{G}_q(k,t)] \asymp \frac{C}{(1-q)\sqrt{\log_2 x}} +\theta\sqrt{\log _2 x}+\frac{D(t)}{\sqrt{\log_2 x}}
\end{equation}
where $C$ is defined in \eqref{deinition_of_c}.
\end{lemma}

\begin{proof}[ Assuming Lemma \ref{Key3_1} let's prove Lemma \ref{key3}.]
We easily verify in both cases that 
$$
\frac{1}{20} \leqslant  \sum_{x^{{\rm e}^{-(u_m+2)}} < p \leqslant  x^{{\rm e}^{-(u_m+1)}}} \frac{1+\cos(2v_m(t)\log p)}{2p} \leqslant  20
$$
and
$$
\frac{1}{20} \leqslant  \sum_{x^{{\rm e}^{-(u_m+2)}} < p \leqslant  x^{{\rm e}^{-(u_m+1)}}} \frac{1}{2p} \leqslant  20.
$$
By applying Lemmas \ref{Key3_1} and \ref{gaussian_approximation} for $n=R$ and since $R=R(t)\asymp \log_2 x$ for \\$\frac{1}{\sqrt{\log_2 x}} \leqslant |t| \leqslant \frac{1}{2}$, we get
$$
\begin{aligned}
\widetilde{\mathbb{P}}_t[\mathcal{G}_q(k,t)] \ll \min\bigg\{1,\frac{a(q,t)}{\sqrt{R}} \bigg\} & \asymp \min\bigg\{1,\frac{C}{(1-q)\sqrt{\log_2 x}} +\frac{D(t)}{\sqrt{\log_2 x}} \bigg\} \\ & \asymp  \frac{C}{(1-q)\sqrt{\log_2 x}} +\theta\sqrt{\log _2 x}+\frac{D(t)}{\sqrt{\log_2 x}}.
\end{aligned}
$$
For the lower bound, we have from \eqref{equation_P_asymp} and \cite[Probability Result]{Harper}
$$
\mathbb{P}\bigg[ -2j -O(1) \leqslant \sum_{m=1}^j G_m \leqslant \min\{ a(q,t),j\}+2 \log j+O(1) \text{ for all } k\leqslant j \leqslant  R  \bigg] \gg  \min\bigg\{1,\frac{a(q,t)}{\sqrt{R}} \bigg\}.
$$
This gives \eqref{equation_P_asymp}.
\end{proof}
\begin{remark}
    The lower bound \eqref{lower_bound_for_part_2} will be needed for Section 4.2.
\end{remark}
\noindent To prove the Lemma \ref{Key3_1}, we need the following lemma.
\begin{lemma}\label{Key_3_2}
Let $f$ be a Rademacher or Steinhaus multiplicative function. Recall $v_m(t)$ and $u_m(t)$ is defined in \eqref{def_v_m} and \eqref{def_u_m}. For \\$0 \leqslant~  l_m \leqslant~1/80 \sqrt{\log x_{u_m}} $ and $\frac{1}{\sqrt{\log_2 x}} \leqslant |t| \leqslant \delta^5$, we have
$$
\begin{aligned}
    \widetilde{\mathbb{P}}_t\bigg[ l_m\leqslant \log\bigg|I_{u_m}\bigg(&\frac{k}{\log x},v_m(t)\bigg) \bigg| \leqslant  l_m + \frac{1}{m^2}\bigg] \\& = \bigg(1+ O\bigg(\frac{1}{x_{u_m}^{1/60}}\bigg) \bigg)\mathbb{P}\bigg[ l_m \leqslant N_m \leqslant  l_m + \frac{1}{m^2} \bigg]
\end{aligned}
$$
where $(N_m)_{1 \leqslant  m \leqslant  R}$ are independent Gaussian random variables where $N_m$, \\in Rademacher case, has a mean $\sum_{x_{u_m}^{1/{\rm e}} < p \leqslant  x_{u_m}} \frac{c(t,v_m(t),p)}{p} $ where
$$
c(t,v_m(t),p):= 2\cos(t \log p)\cos (v_m(t) \log p )-\frac{1}{2} \cos (2 v_m(t)\log p)
$$
and variance $\sum_{x_{u_m}^{1/{\rm e}} < p \leqslant  x_{u_m}} \frac{1+\cos(2v_m(t)\log p)}{2p}$.\\
In Steinhaus case $N_m$ has mean $\sum_{x_{u_m}^{1/{\rm e}} < p \leqslant  x_{u_m}} \frac{\cos(v_m(t)\log p)}{p} $ and variance \\ $\sum_{x_{u_m}^{1/{\rm e}} < p \leqslant  x_{u_m}} \frac{1}{2p}$.
\end{lemma}
\begin{proof} For Rademacher case, the characteristic function of Gaussain is given by
$$
\mathbb{E}[{\rm e}^{iuN_m}]= \exp \bigg\{ \sum_{x_{u_m}^{1/{\rm e}} < p \leqslant  x_{u_m}} \frac{iuc(t,v_m(t),p) + (u^2/4)(1+\cos(2v_m(t)\log p)) }{p} \bigg\}.
$$
For Steinhaus case, the characteristic function of Gaussain is given by
$$
\mathbb{E}[{\rm e}^{iuN_m}]= \exp \bigg\{ \sum_{x_{u_m}^{1/{\rm e}} < p \leqslant  x_{u_m}} \frac{iu\cos(v_m(t) \log p) + u^2/4 }{p} \bigg\}.
$$
By \cite[lemma 1]{Harper} and \cite[lemma 2]{Harper}, for $|u| \leqslant  x_{u_m}^{1/20}$, we get easily in both cases
$$
\begin{aligned}
    & \bigg|\widetilde{\mathbb{E}}_t \bigg[{\rm e}^{iu\log|I_{u_m}(\frac{k}{\log x},v_m(t))|}\bigg] -\mathbb{E}\bigg[{\rm e}^{iuN_m}\bigg]\bigg|
    \\ & \ll \exp \bigg\{ -\frac{u^2}{4} \!\!\!\!\!  \sum_{x_{u_m}^{1/{\rm e}} < p \leqslant  x_{u_m}}\!\!\!\!\! \frac{1}{p}+\frac{(1+a_f)\cos(2v_m(t))}{2p} \bigg\}\frac{|u|+|u|^3}{\sqrt{x_{u_m}^{1/{\rm e}}}\log x_{u_m}}.
\end{aligned}
$$
By Berry-Esseen theorem (see e.g. \cite[Theorem 7.6.1]{Gut}), we have
$$
\begin{aligned}
&\bigg|\widetilde{\mathbb{P}}\left[ l_m \leqslant \log \left|I_{u_m}\left(\frac{k}{\log x}, v_m(t)\right)\right| \leqslant  l_m +\frac{1}{m^2} \right]- \mathbb{P}\left[ l_m \leqslant N_m \leqslant l_m + \frac{1}{m^2}\right] \bigg|
\\ \ll & \int_{-x_{u_m}^{1 / 20}}^{x_{u_m}^{1 / 20}}\left|\frac{\widetilde{\mathbb{E}} \big[ {\rm e}^{i u \log \left|I_{u_m}\left(\frac{k}{\log x}, v_m(t)\right)\right|} \big]-\mathbb{E} \big[{\rm e}^{i u N_m}\big]}{u}\right| {\rm d} u+\frac{1}{x_{u_m}^{1 / 20}} \\
\ll & \int_{-x_{u_m}^{1 / 20}}^{x_{u_m}^{1 / 20}} \frac{1+u^2}{\sqrt{x_{u_m}^{1/{\rm e}}}\log x_{u_m}} {\rm d} u+\frac{1}{x_{u_m}^{1 / 20}} \ll \frac{1}{x_{u_m}^{1 / 60}} .
\end{aligned}
$$
This ends the proof.
\end{proof}
\begin{proof}[Proof of Lemma \ref{Key3_1}.]
First we prove the following inequality: Let $(\Tilde{v}_j)_{1\leqslant j\leqslant R}$ and $(\Tilde{u}_j)_{1\leqslant j\leqslant R}$ be a sequences verifying
$ - 1/80 \sqrt{\log x_j} \leqslant  \Tilde{u}_j \leqslant \Tilde{v}_j \leqslant  1/80 \sqrt{\log x_j}$ for all $1\leqslant j\leqslant R$. We have uniformly for $1/\sqrt{\log_2 x} \leqslant  |t|\leqslant  \delta^5$ and $ 0 \leqslant k \leqslant K$
\begin{equation}
    \begin{aligned}\label{inequality_approx_sum_N}
& \mathbb{P}\bigg[ \Tilde{u}_j +2 \leqslant \sum_{m=1}^j  N_m \leqslant  \Tilde{v}_j - 2 \text{ for all } k \leqslant  j \leqslant  R \bigg]  
\\& \ll \widetilde{\mathbb{P}}_t\bigg[ \Tilde{u}_j \leqslant  \sum_{m=1}^j  \log\bigg|I_{u_m}\bigg(\frac{k}{\log x},v_m(t)\bigg) \bigg| \leqslant  \Tilde{v}_j  \text{ for all } k \leqslant  j \leqslant  R \bigg]  \\ & \ll  \mathbb{P}\bigg[ \Tilde{u}_j -2 \leqslant \sum_{m=1}^j  N_m \leqslant  \Tilde{v}_j + 2 \text{ for all } k \leqslant  j \leqslant  R \bigg]
\end{aligned}
\end{equation}
where $(N_m)_{1 \leqslant  m \leqslant  R}$ are independent Gaussian random variables where $N_m$, in Rademacher case, has a mean\\ $\sum_{x_{u_m}^{1/{\rm e}} < p \leqslant  x_{u_m}} \frac{c(t,v_m(t),p)}{p} $ where
$$
c(t,v_m(t),p):= 2\cos(t \log p)\cos (v_m(t) \log p )-\frac{1}{2} \cos (2 v_m(t)\log p)
$$
and variance $\sum_{x_{u_m}^{1/{\rm e}} < p \leqslant  x_{u_m}} \frac{1+\cos(2v_m(t)\log p)}{2p}$.\\
In Steinhaus case $N_m$ has mean $\sum_{x_{u_m}^{1/{\rm e}} < p \leqslant  x_{u_m}} \frac{\cos(v_m(t)\log p)}{p} $ and variance $\sum_{x_{u_m}^{1/{\rm e}} < p \leqslant  x_{u_m}} \frac{1}{2p}$.\\
We define $\mathcal{R}_m:= \big\{ r \in 1/m^2 \mathbb{Z}\! : \!  r \leqslant  1/40 \sqrt{ \log x_m} +2  \big\}$, then in order to have \\$ \Tilde{u}_j \leqslant~ \sum_{m=1}^j  \log\big|I_{u_m}\big(\frac{k}{\log x},v_m(t)\big) \big| \leqslant  \Tilde{v}_j \text{ for all } k \leqslant  j \leqslant  R$, we must have
$$
r_m\leqslant  \log\bigg|I_{u_m}\bigg(\frac{k}{\log x},v_m(t)\bigg) \bigg| \leqslant  r_m +1/m^2\text{ for all } 1 \leqslant  m \leqslant  R
$$
where $r_m \in \mathcal{R}_m$ for all $1 \leqslant  m \leqslant  R$, with $\Tilde{u}_j - \sum_{m=1}^j \frac{1}{m^2} \leqslant \sum_{m=1}^j r_m \leqslant  \Tilde{v}_j  $ for all $k \leqslant  j \leqslant  R$. Thus, we have, in both cases, by Lemma \ref{Key_3_2}
$$
\begin{aligned}
& \widetilde{\mathbb{P}}_t\bigg[\Tilde{u}_j \leqslant \sum_{m=1}^j  \log\bigg|I_{u_m}\bigg(\frac{k}{\log x},v_m(t)\bigg) \bigg| \leqslant  \Tilde{v}_j \text{ for all } k \leqslant  j \leqslant  R \bigg]
\\ & \leqslant  \widetilde{\mathbb{P}}_t\bigg[  \!\! \!\!\!\!  \bigcup_{\substack{ r_1 \in \mathcal{R}_1 ... r_R \in \mathcal{R}_R \\ \sum_{m=1}^j r_m \leqslant  \Tilde{v}_j \forall k \leqslant  j \leqslant  R }}  \! \!\!\! \!\! \!\!\!\! \bigg\{r_m \leqslant  \log\bigg|I_{u_m}\bigg(\frac{k}{\log x},v_m(t)\bigg) \bigg| \leqslant  r_m + 1/m^2 \text{ for all } 1 \leqslant  m \leqslant  R \bigg\}\bigg]
\\ & \leqslant  \! \!\!\! \!\! \!\!\!  \sum_{\substack{ r_1 \in \mathcal{R}_1 ... r_R \in \mathcal{R}_R \\ \sum_{m=1}^j r_m \leqslant  \Tilde{v}_j \forall k \leqslant  j \leqslant  R }}  \widetilde{\mathbb{P}}_t\bigg[   r_m \leqslant  \log\bigg|I_{u_m}\bigg(\frac{k}{\log x},v_m(t)\bigg) \bigg| \leqslant  r_m + 1/m^2 \text{ for all } 1 \leqslant  m \leqslant  R \bigg]
\\ & \ll   \sum_{\substack{ r_1 \in \mathcal{R}_1 ... r_R \in \mathcal{R}_R \\ \sum_{m=1}^j r_m \leqslant  \Tilde{v}_j \forall k \leqslant  j \leqslant  R }}  \mathbb{P}\bigg[  r_m \leqslant  N_m \leqslant  r_m + 1/m^2 \text{ for all } 1 \leqslant  m \leqslant  R\bigg]. 
\end{aligned}
$$
Since for each $(r_1  ... r_R) \neq (r'_1  ... r'_R)$ from $ \mathcal{R}_1 \times ... \times \mathcal{R}_R $ , we have the probability of the intersection of the events $ r_m \leqslant  N_m  \leqslant  r_m + 1/m^2 \text{ for all } 1 \leqslant  m \leqslant  R$ and $r'_m \leqslant  N_m \leqslant  r'_m + 1/m^2 \text{ for all } 1 \leqslant  m \leqslant  R$ is equal to $0$, we get then
$$
\begin{aligned}
&\!\!\!\!\!\!\!\!\!\!\!\!\!\!\!\!\!\!\!\!\sum_{\substack{ r_1 \in \mathcal{R}_1 ... r_R \in \mathcal{R}_R \\ \sum_{m=1}^j r_m \leqslant  \Tilde{v}_j \forall k \leqslant  j \leqslant  R }}  \mathbb{P}\bigg[  r_m \leqslant  N_m \leqslant  r_m + 1/m^2 \text{ for all } 1 \leqslant  m \leqslant  R\bigg]
\\ &\!\!\!\!\!\!\!\!\!\!\!\!\!\!\!\!\!\!\!\! =  \mathbb{P}\bigg[ \bigcup_{\substack{ r_1 \in \mathcal{R}_1 ... r_R \in \mathcal{R}_R \\ \sum_{m=1}^j r_m \leqslant  \Tilde{v}_j \forall k \leqslant  j \leqslant  R }} \bigg\{ r_m \leqslant  N_m \leqslant  r_m + 1/m^2 \text{ for all } 1 \leqslant  m \leqslant R \bigg\}\bigg]
\\ &\!\!\!\!\!\!\!\!\!\!\!\!\!\!\!\!\!\!\!\! \leqslant\mathbb{P}\bigg[ \Tilde{u}_j -2 \leqslant  \sum_{m=1}^j N_m \leqslant  \Tilde{v}_j + 2 \text{ for all } k \leqslant  j \leqslant  R \bigg].
\end{aligned}
$$
Recall that we have $|t-v_m(t)|\ll \frac{1}{\log x_{u_m} \log_2 x_{u_m}} $. In the Rademacher case, we have 
$$
\begin{aligned}
 2\cos\big(t \log p\big)\cos \big(v_m(t) \log p  \big) & =  \cos\big((t+v_m(t)) \log p\big)+\cos\big((t-v_m(t)) \log p\big)
 \\&= \cos\big((t+v_m(t)) \log p\big) + 1 + O(1/m^2).
\end{aligned}
$$
Since $|t|\log x_{R} \gg {\rm e}^{B} $ (recall that $R$ is defined in \eqref{def_of_S_R}), we get from Prime Number Theorem
$$
\begin{aligned}
\sum_{x_{u_m}^{1 / {\rm e}}<p \leqslant x_{u_m}} \frac{\cos \left(\left(t+v_m(t)\right) \log p\right)}{p} \,\, ,  \sum_{x_{u_m}^{1 / e}<p \leqslant x_{u_m}} \frac{\cos \left(2 v_m(t) \log p\right)}{p} & \ll \frac{1}{|t| \log x_{u_m}}
\\ & \ll \frac{1}{{\rm e}^m|t| \log x_{R}} 
\\ & \ll \frac{1}{B {\rm e}^m}.
\end{aligned}
$$
Since 
$$
\sum_{x_{u_m}^{1/{\rm e}} < p \leqslant  x_{u_m}} \frac{1}{p} = 1 + O\bigg(\frac{1}{{\rm e}^{b_0 {\rm e} ^{m/2}}}\bigg)
$$
where $b_0$ is an absolute constant, we finally get
$$
\sum_{x_{u_m}^{1/{\rm e}} < p \leqslant  x_{u_m}} \frac{c(t,v_m(t),p)}{p} = 1 + O\bigg(\frac{1}{m^2}\bigg).
$$
By taking 
$$
\begin{aligned}
\Tilde{v}_k & = a_k(q,t) 
\end{aligned}
$$
in \eqref{inequality_approx_sum_N}, 
and $$G_m=N_m - \sum_{x_{u_m}^{1/{\rm e}} < p \leqslant  x_{u_m}} \frac{c(t,v_m(t),p)}{p}, $$
we get 
$$
\begin{aligned}
& \widetilde{\mathbb{P}}_t\bigg[ \sum_{m=1}^j  \log\bigg|I_{u_m}\bigg(\frac{k}{\log x},v_m(t)\bigg) \bigg| \leqslant   a_j(q,t)  \text{ for all } k \leqslant  j \leqslant  R \bigg] 
\\ & \ll  \mathbb{P}\bigg[ \sum_{m=1}^j  G_m \leqslant   j - \sum_{m=1}^j \sum_{x_{u_m}^{1/{\rm e}} < p \leqslant  x_{u_m}} \frac{c(t,v_m(t),p)}{p} + a(q,t) +2 \log j \text{ for all } k \leqslant  j \leqslant  R \bigg]
\\ & \ll \mathbb{P}\bigg[ \sum_{m=1}^j  G_m \leqslant   a(q,t) +2 \log j + O(1) \text{ for all } k \leqslant  j \leqslant  R \bigg].
\end{aligned}
$$
The proof of the corresponding lower bound of \eqref{inequality_approx_sum_N} is exactly similar. By following the same steps, we have
$$
\begin{aligned}
    & \widetilde{\mathbb{P}}_t\bigg[  -a_j(q,t) \leqslant \sum_{m=1}^j  \log\big|I_{u_m}\big(\frac{k}{\log x},v_m(t)\big) \big| \leqslant   a_j(q,t)  \text{ for all } k \leqslant  j \leqslant  R \bigg] 
    \\ & \gg \mathbb{P}\bigg[ -j - \sum_{m=1}^j \sum_{x_{u_m}^{1/{\rm e}} < p \leqslant  x_{u_m}} \frac{c(t,v_m(t),p)}{p}  \leqslant \sum_{m=1}^j G_m \leqslant  a(q,t)+2 \log j \text{ for all } k\leqslant j \leqslant  R  \bigg].
    \\ & \gg \mathbb{P}\bigg[ -2j +O(1) \leqslant \sum_{m=1}^j G_m \leqslant  a(q,t)+2 \log j \text{ for all } k\leqslant j \leqslant  R  \bigg]
\end{aligned}
$$
For Steinhaus case, we have 
$$
\begin{aligned}
& \widetilde{\mathbb{P}}_t\bigg[ \sum_{m=1}^j  \log\big|I_{u_m}\big(\frac{k}{\log x},v_m(t)\big) \big| \leqslant  a_j(q,t) \text{ for all } k \leqslant  j \leqslant  R \bigg] 
\\ & \ll  \mathbb{P}\bigg[ \sum_{m=1}^j  G_m \leqslant    - \sum_{m=1}^j \sum_{x_{u_m}^{1/{\rm e}} < p \leqslant  x_{u_m}} \frac{1}{p} + a_j(q,t)  \text{ for all } k \leqslant  j \leqslant  R \bigg]
\\ & \ll \mathbb{P}\bigg[ \sum_{m=1}^j  G_m \leqslant   a(q,t) +2 \log j + O(1) \text{ for all } k \leqslant j \leqslant  R \bigg].
\end{aligned}
$$
For the lower bound in the Steinhaus case
$$
\begin{aligned}
    &\widetilde{\mathbb{P}}_t[  -a_j(q,t) \leqslant \sum_{m=1}^j  \log\big|I_{u_m}\big(\frac{k}{\log x},v_m(t)\big) \big| \leqslant   a_j(q,t)  \text{ for all } k \leqslant  j \leqslant  R ] 
        \\ & \gg \mathbb{P}\bigg[ -j -  \sum_{m=1}^j \sum_{x_{u_m}^{1/{\rm e}} < p \leqslant  x_{u_m}} \frac{1}{p}  \leqslant \sum_{m=1}^j G_m \leqslant  a(q,t)+2 \log j \text{ for all } k\leqslant j \leqslant  R  \bigg]
    \\ & \gg \mathbb{P}\bigg[ -2j +O(1) \leqslant \sum_{m=1}^j G_m \leqslant  a(q,t)+2 \log j \text{ for all } k\leqslant j \leqslant  R  \bigg].
\end{aligned}
$$
This ends the proof.
\end{proof}

\subsection{The proof of the lower bound.}
The approach used to prove Theorem \ref{theoreme_principal_2} is once again derived from \cite{Harper}. The initial phase of our argument involves a reduction step aimed at narrowing down the initial expression to comprehending specific averages of random Euler products, much like the methodology employed in the upper bound demonstration.\\
In the following, we will use the notations $M,\, M^{+},\, M^{-}$ defined in \eqref{def_M} and \eqref{def_M+_M-}. Let first start when $y \ll \frac{x}{\log x}$. We have the following Lemma
\begin{lemma}\label{ Xu_sound_lemma}
    Let $f$ be a Steinhaus or Rademacher multiplicative function. Let $x$ and $y$ large, with $y\ll \frac{x}{\log x}$. We have uniformly for $0 \leqslant q \leqslant 1$
    $$ \mathbb{E}\bigg[\big|M(x,y;f) \big|^{2q}\bigg] \asymp y^q.$$
\end{lemma}
\begin{proof}
    Is a direct result from the fact that for $y \ll \frac{x}{\log x}$ the quantity $\frac{M(x,y;f)}{\mathbb{E}\big[\big|M(x,y;f) \big|^{2}\big]} $ is distributed like a standard normal random variable with mean $0$ and variance $1$ (See for instance \cite[Corollary 1.2 and Theorem 9.1]{Sound_Xu}).
\end{proof}
\noindent Now, we can reduce the problem on studying $M(x,y;f)$ when $y \geqslant \frac{x}{\log x}$.
Let $\epsilon$ be an  auxiliary Rademacher random variable, we have 
$$
\begin{aligned}
    \mathbb{E}\bigg[\big| M^{-}(x,y,\sqrt{2x};f) \big|^{2q}\bigg] & = \frac{1}{2^{2q}} \mathbb{E}\bigg[\big| M^{-}\big(x,y,\sqrt{2x};f\big) + M^{+}\big(x,y,\sqrt{2x};f\big)\\ & \,\,\,\,\,\,\,\,\,+ M^{-}\big(x,y,\sqrt{2x};f\big) - M^{+}\big(x,y,\sqrt{2x};f\big) \big|^{2q}\bigg]
    \\ & \leqslant \mathbb{E}\bigg[\big| M^{-}\big(x,y,\sqrt{2x};f\big) + M^{+}\big(x,y,\sqrt{2x};f\big) \big|^{2q}\bigg] \\ & \,\,\,\,\,\,\,\,\,+ \mathbb{E}\bigg[\big| M^{-}\big(x,y,\sqrt{2x};f\big) - M^{+}\big(x,y,\sqrt{2x};f\big) \big|^{2q}\bigg]
    \\ & = 2 \mathbb{E}\bigg[\big|M(x,y;f) \big|^{2q}\bigg].
\end{aligned}
$$
Here, we used the fact that the law of $ - M^{+}(x,y,\sqrt{2x};f) = - \sum_{p > \sqrt{2x}} f(p) M(x/p,y/p;f)$ conditional on the values $(f(p))_{p\leqslant \sqrt{2x}}$ is the same as the law of $M^{+}(x,y,\sqrt{2x};f)$. Thus we have
$$
\mathbb{E}\bigg[\big| M(x,y;f) \big|^{2q}\bigg] \gg \mathbb{E}\bigg[\big| M^{-}(x,y,\sqrt{2x};f) \big|^{2q}\bigg].
$$
In the decomposition 
$$
\begin{aligned}
M^{-}(x,y,\sqrt{2x};f) & = \sum_{ \sqrt{2x} <p\leqslant x+y} f(p) M\bigg(\frac{x}{p},\frac{y}{p};f\bigg)
\\ & = \sum_{ y <p\leqslant x+y} f(p)  M\bigg(\frac{x}{p},\frac{y}{p};f\bigg) + \sum_{ \sqrt{2x} <p\leqslant y} f(p)  M\bigg(\frac{x}{p},\frac{y}{p};f\bigg),
\end{aligned}
$$
the inner sum in both sums which are determined by $(f(p))_{p> \sqrt{2x}}$ are independent from $ \sum_{\substack{\frac{x}{p}<n\leqslant \frac{x+y}{p}}}f(n)$, by applying the Khintchine’s inequality, we get
$$
\begin{aligned}
    \mathbb{E}\bigg[\big| M^{-}(x,y,\sqrt{2x};f) \big|^{2q}\bigg] & \gg \mathbb{E}\bigg[\bigg(\sum_{y <p }\bigg| M^{+}\bigg(\frac{x}{p},\frac{x+y}{p},\sqrt{2x};f\bigg) \bigg|^{2}
    \\ &\,\,\,\,\,\,\,\,\,\,\,+\sum_{\sqrt{2x}<p\leqslant y}\bigg| M^{+}\bigg(\frac{x}{p},\frac{x+y}{p},\sqrt{2x};f\bigg) \bigg|^{2}\bigg)^q\bigg] 
    \\ & = \mathbb{E}\bigg[\bigg(\sum_{\substack{x< n \leqslant x+y \\ P(n)>y}}1+\sum_{\sqrt{2x}<p\leqslant y}\bigg| M^{+}\bigg(\frac{x}{p},\frac{x+y}{p},\sqrt{2x};f\bigg) \bigg|^{2}\bigg)^q\bigg].
    \end{aligned}
$$
\begin{remark}
   Note that the left-hand side of the expression above is deterministic. One can prove that for any $\varepsilon >0$ and for $\sqrt{x} \leqslant y \leqslant x^{1-\varepsilon}$, the following holds:
    $$
    \sum_{\substack{x< n \leqslant x+y \\ P(n)>y}}1 \gg_{\varepsilon} y.
    $$
    However, this result is irrelevant as we are investigating $M(x,y;f)$ in cases where $y \geqslant ~\frac{x}{\log x}$.
\end{remark}
\noindent In the following, we will focus on $ \sum_{\sqrt{2x}<p\leqslant y}\big| M^{+}\big(\frac{x}{p},\frac{x+y}{p},\sqrt{2x};f\big) \big|^{2}$. We start with this following Lemma.
\begin{lemma}\label{lemma_important1}
Let $f$ be a Rademacher or Steinhaus multiplicative function. There exists a large absolute constant $C_1 > 0$ such that the following is true. Let $x$ be large enough and assume that $y\geqslant \frac{x}{\log x}$. For any large quantity $V \leqslant (\log x )^{0.1}$. We have uniformly for all $2/3\leqslant q\leqslant 1$ 
    \begin{equation}\label{eq_proof_important1}
    \begin{aligned}
            \mathbb{E}\bigg[\big| M^{-}(x,y,\sqrt{2x};f) \big|^{2q}\bigg] \gg & \mathbb{E}\bigg[ \bigg(\frac{y}{\delta\log x} \int_{-\delta}^{\delta} \big|F_{\sqrt{2x}}(it +4V/\log x)\big|^2 {\rm d}t\bigg)^q \bigg]\\ &  -C_1y^q \bigg( \bigg(  \frac{1}{V{\rm e}^{V}}\bigg)^q-\frac{\log_2 x}{\log x}\bigg).
    \end{aligned}
    \end{equation}
\end{lemma}
    \begin{proof}
Note first that 
$$
\mathbb{E}\bigg[\big| M^{-}(x,y,\sqrt{2x};f) \big|^{2q}\bigg] \gg \mathbb{E}\bigg[\bigg(\sum_{\sqrt{2x}<p\leqslant y}\bigg| M^{-}\bigg(\frac{x}{p},\frac{x+y}{p},\sqrt{2x};f\bigg) \bigg|^{2}\bigg)^q\bigg].
$$
We follow the same techniques as in Harper \cite{Harper} in the ``Proof of Proposition 3 and 4". Let $X \geqslant {\rm e}^{\sqrt{\log x}}$. It can be readily checked that the following inequality consistently holds: $|x+ a + b|^2 \geqslant (\frac{1}{4})|x|^2 - \min\{2(|a|^2+|b|^2), |x/2|^2\}$. Consequently, 
$$
\begin{aligned}
    & \sum_{\sqrt{2x} < p\leqslant y} \bigg|M^{+}\bigg(\frac{x}{p},\frac{x+y}{p},\sqrt{2x};f\bigg)\bigg|^2  \\ & = \sum_{\sqrt{2x} < p\leqslant y}  \frac{X}{p} \int^{p(1+1/X)}_{p} \bigg|M^{+}\bigg(\frac{x}{p},\frac{x+y}{p},\sqrt{2x};f\bigg)\bigg|^2{\rm d}t
    \\ &  \geqslant \sum_{\sqrt{2x} < p\leqslant y}  \frac{X}{p} \int^{p(1+1/X)}_{p} \bigg|M^{+}\bigg(\frac{x}{t},\frac{y}{t},\sqrt{2x};f\bigg)\bigg|^2{\rm d}t
    \\ & \,\,\,\, -  \sum_{\sqrt{2x} < p\leqslant y}  \frac{X}{p} \int^{p(1+1/X)}_{p} \min\bigg\{ 2\bigg|M^{+}\bigg(\frac{x+y}{t},\frac{x+y}{p},\sqrt{2x};f\bigg)\bigg|^2{\rm d}t \\ & \,\,\,\,\,\,\,+ 2 \bigg|M^{+}\bigg(\frac{x}{t},\frac{x}{p},\sqrt{2x};f\bigg)\bigg|^2, \frac{1}{4} \bigg|M^{+}\bigg(\frac{x}{t},\frac{y}{t},\sqrt{2x};f\bigg)\bigg|^2 \bigg\} {\rm d}t.
\end{aligned}
$$
thus for any $2/3\leqslant q \leqslant 1 $, we get
$$
\begin{aligned}
    &\mathbb{E}\bigg[\big| M^{-}(x,y,\sqrt{2x};f)\big|^{2q}\bigg]\\ &\gg  \frac{1}{(\log x)^q} \mathbb{E}\bigg[\bigg( \frac{1}{4} \sum_{\sqrt{2x}<p\leqslant y} \log p \frac{X}{p}\int_{p}^{p(1+1/X)} \bigg|M^{+}\bigg(\frac{x}{t},\frac{y}{t},\sqrt{2x};f\bigg)\bigg|^2{\rm d}t\bigg)^q\bigg]
    \\& - \frac{1}{(\log x)^q} \mathbb{E}\bigg[\bigg( 2 \sum_{\sqrt{2x}<p\leqslant y} \log p \frac{X}{p}\int_{p}^{p(1+1/X)} \bigg|\sum_{\substack{\frac{x}{t}<n\leqslant \frac{x}{p} \\ P(n)>\sqrt{2x}}} f(n)\bigg|^2{\rm d}t\bigg)^q\bigg]
    \\& - \frac{1}{(\log x)^q} \mathbb{E}\bigg[\bigg( 2 \sum_{\sqrt{2x}<p\leqslant y} \log p \frac{X}{p}\int_{p}^{p(1+1/X)} \bigg|\sum_{\substack{\frac{x+y}{t}<n\leqslant \frac{x+y}{p} \\ P(n)>\sqrt{2x}}} f(n)\bigg|^2{\rm d}t\bigg)^q\bigg].
\end{aligned}
$$
We deal first with the last two terms. By H\"{o}lder's inequality and by using the same techniques as in the inequality \eqref{inequality_0101}, we have
$$
\begin{aligned}
    & \mathbb{E}\bigg[\bigg(  \sum_{\sqrt{2x}<p\leqslant y} \log p \frac{X}{p}\int_{p}^{p(1+1/X)} \bigg|\sum_{\substack{\frac{x+y}{t}<n\leqslant \frac{x+y}{p} \\ P(n)>\sqrt{2x}}} f(n)\bigg|^2{\rm d}t\bigg)^q\bigg]
    \\ &\leqslant \bigg(  \sum_{\sqrt{2x}<p\leqslant y} \log p \frac{X}{p}\int_{p}^{p(1+1/X)} \mathbb{E}\bigg[\bigg|\sum_{\substack{\frac{x+y}{t}<n\leqslant \frac{x+y}{p} \\ P(n)>\sqrt{2x}}} f(n)\bigg|^2\bigg]{\rm d}t\bigg)^q
    \\ & \ll \bigg(  \sum_{\sqrt{2x}<p\leqslant y} \log p \bigg(\frac{x}{pX}+1\bigg)\bigg)^q\ll \bigg(\frac{x\log x}{X}+y\bigg)^q\ll y^q\bigg( \log x\frac{\delta}{X}+1\bigg)^q.
\end{aligned}
$$
Recall that $\delta \ll \log x$ and $X \geqslant {\rm e}^{\sqrt{\log x}}$. We obtain
$$
\mathbb{E}\bigg[\bigg(\frac{1}{\log x} \sum_{\sqrt{2x}<p\leqslant x+y} \log p \frac{X}{p}\int_{p}^{p(1+1/X)} \bigg|\sum_{\substack{\frac{x+y}{t}<n\leqslant \frac{x+y}{p} \\ P(n)>\sqrt{2x}}} f(n)\bigg|^2{\rm d}t\bigg)^q\bigg] \ll \bigg( \frac{y}{\log x}\bigg)^q.
$$
We prove similarly that
$$
\mathbb{E}\bigg[\bigg(\frac{1}{\log x} \sum_{\sqrt{2x}<p\leqslant x+y} \log p \frac{X}{p}\int_{p}^{p(1+1/X)} \bigg|\sum_{\substack{\frac{x}{t}<n\leqslant \frac{x}{p} \\ P(n)>\sqrt{2x}}} f(n)\bigg|^2{\rm d}t\bigg)^q\bigg] \ll \bigg(\frac{y}{\log x}\bigg)^q.
$$
Let deal now with $\sum_{\sqrt{2x}<p\leqslant y} \log p \frac{X}{p}\int_{p}^{p(1+1/X)} \big|M^{+}\big(\frac{x}{t},\frac{y}{t},\sqrt{2x};f\big)\big|^2 {\rm d}t$.\\
We have 
$$
\begin{aligned}
    &\sum_{\sqrt{2x}<p\leqslant y} \log p \frac{X}{p}\int_{p}^{p(1+1/X)} \bigg|M^{+}\bigg(\frac{x}{t},\frac{y}{t},\sqrt{2x};f\bigg)\bigg|^2{\rm d}t \\& \gg \int_{y}^{\sqrt{2x}} \sum_{\frac{t}{1+\frac{1}{X}}< p <t} \log p \frac{X}{p} \bigg|M^{+}\bigg(\frac{x}{t},\frac{y}{t},\sqrt{2x};f\bigg)\bigg|^2 {\rm d}t
    \\ & \gg \int^{y}_{\sqrt{2x}} \bigg|M^{+}\bigg(\frac{x}{t},\frac{y}{t},\sqrt{2x};f\bigg)\bigg|^2 {\rm d}t
     = y  \int^{y/\sqrt{2x}}_{1} \big|M^{+}(\delta z,z,\sqrt{2x};f)\big|^2 \frac{{\rm d}z}{z^2}.
\end{aligned}
$$
We used the fact that $\sum_{\frac{t}{1+\frac{1}{X}}< p <t} \log p \frac{X}{p} \asymp 1$ followed by a change of variable $z=y/t$. Let's focus on $ \int^{y/\sqrt{2x}}_{1} \big|M^{+}(\delta z,z,\sqrt{2x};f)\big|^2 \frac{{\rm d}z}{z^2}$, we have
$$
\begin{aligned}
& \int^{y/\sqrt{2x}}_{1} \big|M^{+}(\delta z,z,\sqrt{2x};f)\big|^2 \frac{{\rm d}z}{z^2} \\ & = \int^{y/\sqrt{2x}}_{0} \big|M^{+}(\delta z,z,\sqrt{2x};f)\big|^2 \frac{{\rm d}z}{z^2} - \int^{1}_{0} \big|M^{+}(\delta z,z,\sqrt{2x};f)\big|^2 \frac{{\rm d}z}{z^2}.
\end{aligned}
$$
Note that $$\mathbb{E}\bigg[\frac{1}{\log x}\int^{1}_{0} \big|M^{+}(\delta z,z,\sqrt{2x};f)\big|^2 \frac{{\rm d}z}{z^2}\bigg] \leqslant \frac{1}{\log x}\sum_{1\leqslant n\leqslant \delta+1 } \frac{1}{n}  \asymp \frac{\log \delta}{\log x}\ll \frac{\log_2 x}{\log x}. $$
Let focus now on $  \int^{y/\sqrt{2x}}_{0} \big|M^{+}(\delta z,z,\sqrt{2x};f)\big|^2 \frac{{\rm d}z}{z^2}$.
Let $V$ be a large number such that $V<(\log x)^{0.1}$, we have for $x$ large enough
\begin{equation}\label{ineq01}
\begin{aligned}
& \mathbb{E}\bigg[\bigg(\int^{y/\sqrt{2x}}_{0} \big|M^{+}(\delta z,z,\sqrt{2x})\big|^2 \frac{{\rm d}z}{z^2} \bigg)^q\bigg] \\ & \geqslant 2^{-q} \mathbb{E}\bigg[\bigg(\int^{y/\sqrt{2x}}_{0} \big|M^{+}(\delta z,z,\sqrt{2x};f)\big|^2 \frac{{\rm d}z}{z^{2+ 8V/\log x}} \bigg)^q\bigg]
\\ & \geqslant \frac{1}{2} \mathbb{E}\bigg[\bigg(\int^{+\infty}_{0} \big|M^{+}(\delta z,z,\sqrt{2x};f)\big|^2 \frac{{\rm d}z}{z^{2+ 8V/\log x}} \bigg)^q\bigg]
\\ & \,\,\,\,\,\,-\frac{1}{2}\mathbb{E}\bigg[\bigg(\int^{+\infty}_{y/\sqrt{2x}} \big|M^{+}(\delta z,z,\sqrt{2x};f)\big|^2 \frac{{\rm d}z}{z^{2+ 8V/\log x}} \bigg)^q\bigg]
\\ & \geqslant \frac{1}{2} \mathbb{E}\bigg[\bigg(\int^{+\infty}_{0} \big|M^{+}(\delta z,z,\sqrt{2x};f)\big|^2 \frac{{\rm d}z}{z^{2+ 8V/\log x}} \bigg)^q\bigg]
\\ & \,\,\,\,\,\,-\frac{1}{2{\rm e}^{Vq}}\mathbb{E}\bigg[\bigg(\int^{+\infty}_{0} \big|M^{+}(\delta z,z,\sqrt{2x};f)\big|^2 \frac{{\rm d}z}{z^{2+ 4V/\log x}} \bigg)^q\bigg].
\end{aligned}
\end{equation}
By using Lemma \ref{parseval}, we have $$ 
\begin{aligned}
&\mathbb{E}\bigg[\bigg(\int^{+\infty}_{0} \big|M^{+}(\delta z,z,\sqrt{2x};f)\big|^2 \frac{{\rm d}z}{z^{2+ 8V/\log x}} \bigg)^q\bigg]
  \gg \mathbb{E}\bigg[ \bigg(\frac{1}{\delta} \int_{-\delta}^{\delta} \big|F_{\sqrt{2x}}(it +4V/\log x)\big|^2 {\rm d}t\bigg)^q \bigg]
\end{aligned}
$$
where $F_{\sqrt{2x}}$ is defined in \eqref{definition_F_z} for $z=\sqrt{2x}$.
Note that we have as well
$$
\begin{aligned}
    \mathbb{E}\bigg[\bigg(\int^{+\infty}_{0} \big|M^{+}(\delta z,z &,\sqrt{2x};f)\big|^2 \frac{{\rm d}z}{z^{2+ 4V/\log x}} \bigg)^q\bigg] \\ & \leqslant \bigg(\int^{+\infty}_{0} \mathbb{E}\bigg[\big|M^{+}(\delta z,z,\sqrt{2x};f)\big|^2 \bigg]\frac{{\rm d}z}{z^{2+ 4V/\log x}} \bigg)^q 
    \\ & \ll \bigg( \frac{\log \delta}{V}+ \frac{\log x}{V}\bigg)^q\ll \bigg(  \frac{\log x}{V}\bigg)^q.
\end{aligned}
$$
At the end, and by gathering all previous inequalities, we deduce that there exists a large constant $C_1$ such that
$$
\begin{aligned}
            \mathbb{E}\bigg[\bigg|\frac{1}{\sqrt{y}}M^{+}\big( x,y,\sqrt{2x};f \big) \bigg|^{2q}\bigg] \gg & \,\,\, \mathbb{E}\bigg[ \bigg(\frac{1}{\delta\log x} \int_{-\delta}^{\delta} \big|F_{\sqrt{2x}}(it +4V/\log x)\big|^2 {\rm d}t\bigg)^q \bigg]\\ &  -  C_1 \bigg( \bigg(  \frac{1}{V{\rm e}^{V}}\bigg)^q - \bigg(\frac{\log_2 x}{\log x}\bigg)^q\bigg).
\end{aligned}
$$
\end{proof}
\noindent Following Harper \cite[section 5.2]{Harper}, we define the following event $L(t)$ by: for all\\ $ B+2 \leqslant k \leqslant \log_2 x -\lfloor\log V\rfloor -3 $, we have
\begin{equation}\label{event_t}
     -Ba_k(q,x) \leqslant  \sum^{k}_{m=1}  \log \left|I_{u_m}\left(\frac{4V}{\log x}, t\right)\right| \leqslant  a_k(q,x)
\end{equation}
where $$a(q,x) := \min\bigg\{ \sqrt{\log_2 x}, \frac{1}{1-q}+ \frac{\theta}{4}\log_2 x  \bigg\} .$$
For Steinhaus case, we define  the set
$$
    \mathcal{L}^{St} := \big\{|t| \leqslant 2 \delta : L(t) \text{ defined by \eqref{event_t}} \text{ holds}  \big\}
$$
and set
\begin{equation}\label{random_set_steinhaus}
\mathcal{S}:= \mathcal{L}^{St}\cap [-\delta,\delta].
\end{equation}
For Rademacher case, we define the event
\begin{equation}\label{random_set_Rademacher}
\mathcal{L}^{Rad} := \big\{\frac{1}{2}\leqslant|t| \leqslant  \delta : L(t) \text{ defined by \eqref{event_t}} \text{ holds}  \big\}.
\end{equation}
We fixe $\mathbb{1}_{L(t)}=0$ for $|t|< 1/2$ in the Rademacher case.
\begin{prop}\label{prop001}
    Let $f$ be a Rademacher or Steinhaus random multiplicative function. Let $\mathcal{H}$ to be $\mathcal{S}$ or $\mathcal{L}^{rad}$. Let $x$ be large enough. Let $y\geqslant\frac{x}{\log x}$  and let $V$ be a large constant. Then uniformly for any $2/3 \leqslant q\leqslant 1$ and $ \theta  \gg \frac{1}{\sqrt{\log_2 x}}$, we have 
    \begin{equation}
        \mathbb{E}\bigg[ \bigg(\frac{1}{\delta} \int_{\mathcal{H}} \big|F_x(it +4V/\log x)\big|^2 {\rm d}t\bigg)^q \bigg] \gg \bigg(\frac{\log x}{V}\bigg)^q.
    \end{equation}
\end{prop}
\begin{proof}[Proof of Theorem \ref{theoreme_principal_2} assuming  Proposition \ref{prop001}.]
By inserting these dual limits into inequality \eqref{eq_proof_important1} of Lemma \ref{lemma_important1} and selecting a sufficiently large constant value V, the factor $C_1/{\rm e}^{V}$ in that expression effectively nullifies the impact of the implicit constants. This ends the proof of Theorem \ref{theoreme_principal_2}.
\end{proof}
\noindent We postpone the proof of Proposition \ref{prop001}, we need first the following proposition.
\begin{prop}\label{prop002}
    Let $f$ be a Rademacher or Steinhaus multiplicative function. Let $\mathcal{H}$ to be $\mathcal{S}$ or $\mathcal{L}^{rad}$. Let $x$ be large enough, $\theta \gg \frac{1}{\sqrt{\log_2 x}}$ and $V$ be a large constant. Then uniformly for any $2/3 \leqslant q\leqslant 1$, we have
    \begin{equation}\label{equation_0101S}
        \bigg(\mathbb{E}\bigg[ \bigg(\frac{1}{\delta} \int_{\mathcal{H}} \big|F_x(it +4V/\log x)\big|^2 {\rm d}t\bigg)^2 \bigg]\bigg)^{1-q} \ll \bigg(\frac{\log x}{V}\bigg)^{2(1-q)}.
    \end{equation}
\end{prop}
    \begin{proof}[Proof of Proposition \ref{prop001} assuming Proposition \ref{prop002}.] In this proof, we will focus  on Steinhaus case, the proof in the Rademacher case is similar.
    By H\"{o}lder's inequality, we have
    $$
    \mathbb{E}\bigg[ \bigg(\frac{1}{\delta} \int_{\mathcal{S}} \big|F_x(it +4V/\log x)\big|^2 {\rm d}t\bigg)^q \bigg]  \geqslant \frac{\bigg(\mathbb{E}\big[ \frac{1}{\delta} \int_{\mathcal{S}} \big|F_x(it +4V/\log x)\big|^2 {\rm d}t \big]\bigg)^{2-q} }{\bigg(\mathbb{E}\big[ \big(\frac{1}{\delta} \int_{\mathcal{S}} \big|F_x(it +4V/\log x)\big|^2 {\rm d}t\big)^2 \big] \bigg)^{1-q}}.
    $$
    We give first a lower bound on the numerator. We have
    $$
    \begin{aligned}
        \mathbb{E}\big[ \frac{1}{\delta} \int_{\mathcal{S}} \big|F_x(it +4V/\log x)\big|^2 {\rm d}t \big] & =  \frac{1}{\delta} \int_{-\delta}^{\delta} \mathbb{E}\big[ \mathbb{1}_{L(t)} \big|F_x(it +4V/\log x)\big|^2 \big] {\rm d}t 
        \\& \geqslant  \frac{1}{\delta} \int_{\substack{\frac{1}{\sqrt{\log_2 x}}<|t|\leqslant \delta }} \mathbb{E}\big[ \mathbb{1}_{L(t)} \big|F_x(it +4V/\log x)\big|^2 \big] {\rm d}t.
    \end{aligned}
    $$
        By using inequality \eqref{equation_P_asymp} in Lemma \ref{Key3_1} and  by taking $R=\lfloor \log_2 x \rfloor -(B+3) - ~\lfloor \log V \rfloor$, \\$a(q,x)=  \min\{\sqrt{\log_2 x},\frac{1}{1-q}+ \frac{\theta}{4}\log_2 x\}$, $h(j)=2\log j$ and since $ \theta \gg \frac{1}{\sqrt{\log_2 x}}$, we get $\widetilde{\mathbb{P}}_t(\mathbb{1}_{L(t)})\gg 1$. We have then
        $$
        \big(  \mathbb{E}\big[ \mathbb{1}_{L(t)} \big|F_x(4V/\log x)\big|^2 \big] \big)^{2-q} \gg \big(   \mathbb{E}\big[ \big|F_x(4V/\log x)\big|^2 \big] \big)^{2-q} \gg \bigg( \frac{\log x}{V}\bigg)^{2-q}.
        $$
        Combining with the \eqref{equation_0101S}, we get then
        $$
        \mathbb{E}\bigg[ \bigg(\frac{1}{\delta} \int_{\mathcal{S}} \big|F_x(it +4V/\log x)\big|^2 {\rm d}t\bigg)^q \bigg] \gg \bigg( \frac{\log x}{V}\bigg)^{q}.
        $$
    \end{proof}

\noindent Before starting the proof of  Proposition \ref{prop002}, we need first the following Lemma.

    \begin{lemma}[Harper]\label{lemma_harper_01}
    Let $f$ be a Steinhaus multiplicative function. Let $x$ be large enough and $V$ be a large constant. Then uniformly for any $2/3 \leqslant q\leqslant 1$ and $\theta \gg \frac{1}{\sqrt{\log_2 x}}$, we have
        \begin{equation}
            \begin{aligned}
                 & \int_{|t|\leqslant 1} \mathbb{E}\big[ \mathbb{1}_{L(0)} 
 \big|F_x(4V/\log x)\big|^2 \mathbb{1}_{L(t)} \big|F_x(it +4V/\log x)\big|^2 \big] {\rm d}t
 \\ & \ll {\rm e}^{2\min\{\sqrt{\log_2 x}, 1/(1-q) +\frac{\theta}{2}\log_2 x \}}\bigg( \frac{\log x}{V}\bigg)^{2}.
            \end{aligned}
        \end{equation}
    \end{lemma}
    \begin{proof}
        Is a direct conclusion from \cite[Key Proposition 5]{Harper}, where  $ \min\{\sqrt{\log_2 x}, \frac{1}{1-q} \}$ is replaced by $\min\{\sqrt{\log_2 x}, \frac{1}{1-q}+ \frac{\theta}{4}\log_2 x \} $.
    \end{proof}
    
\begin{proof}[Proof of Proposition \ref{prop002}.] 
We start with the Steinhaus case.
    By using the fact that $f(n)n^{it}$ has the same law as $f(n)$, we have then
    $$
    \begin{aligned}
        & \mathbb{E}\bigg[ \bigg(\frac{1}{\delta} \int_{\mathcal{S}} \big|F_x(it +4V/\log x)\big|^2 {\rm d}t\bigg)^2 \bigg]
        \\  & = \mathbb{E}\bigg[ \frac{1}{\delta^2} \int_{-\delta}^{\delta} \mathbb{1}_{L(t_1)} \big|F_x(it_1 +4V/\log x)\big|^2 {\rm d}t_1 \int_{-\delta}^{\delta} \mathbb{1}_{L(t_2)} \big|F_x(it_2 +4V/\log x)\big|^2 {\rm d}t_2 \bigg]
        \\ & =  \frac{1}{\delta} \int_{-2\delta}^{2\delta} \mathbb{E}\big[ \mathbb{1}_{L(0)} 
 \big|F_x(4V/\log x)\big|^2 \mathbb{1}_{L(t)} \big|F_x(it +4V/\log x)\big|^2 \big] {\rm d}t.
    \end{aligned}
    $$
    Here, we have made a change of variables $(t=t_1-t_2)$.
    We have for $|t|>1$
    $$
    \begin{aligned}
        & \mathbb{E}\big[ \mathbb{1}_{L(0)} 
 \big|F_x(4V/\log x)\big|^2 \mathbb{1}_{L(t)} \big|F_x(it +4V/\log x)\big|^2 \big] 
 \\ & \leqslant\mathbb{E}\big[ 
 \big|F_x(4V/\log x)\big|^2 \big|F_x(it +4V/\log x)\big|^2 \big] 
 \\ & \ll  \exp\bigg\{ \sum_{p \leqslant x} \frac{2+2\cos(t\log p)}{p^{1+8V/\log x}}  \bigg\} \ll \exp\bigg\{  \sum_{{\rm e}^{1/|t|} < p \leqslant x} \frac{2+2\cos(t\log p)}{p^{1+8V/\log x}} \bigg\}.
    \end{aligned}
    $$
    By applying Lemma \ref{lemma_Xu}, we get uniformly for every $|t|>1$.
    $$
    \begin{aligned}
        \exp\bigg\{  \sum_{{\rm e}^{1/|t|} < p \leqslant x} \frac{2+2\cos(t\log p)}{p^{1+8V/\log x}} \bigg\} \ll \bigg(\frac{\log x}{V}\bigg)^2.
    \end{aligned}
    $$
    Thus
    $$
    \frac{1}{\delta} \int_{1<|t|\leqslant 2\delta} \mathbb{E}\big[ \mathbb{1}_{L(0)} 
 \big|F_x(4V/\log x)\big|^2 \mathbb{1}_{L(t)} \big|F_x(it +4V/\log x)\big|^2 \big] {\rm d}t \ll\bigg(\frac{\log x}{V}\bigg)^2. 
    $$
    Now let's deal with the case where $|t|\leqslant 1$. By applying Lemma \ref{lemma_harper_01}, we have uniformly for all large $x$ and $2/3 \leqslant q\leqslant 1$ and $V$ a large constant
    $$
    \begin{aligned}
            & \frac{1}{\delta} \int_{|t|\leqslant 1} \mathbb{E}\big[ \mathbb{1}_{L(0)} 
 \big|F_x(4V/\log x)\big|^2 \mathbb{1}_{L(t)} \big|F_x(it +4V/\log x)\big|^2 \big] {\rm d}t
 \\ & \ll \frac{1}{\delta}{\rm e}^{2\min\{\sqrt{\log_2 x}, 1/(1-q) +\frac{\theta}{2}\log_2 x \}}\bigg( \frac{\log x}{V}\bigg)^{2}.
    \end{aligned}
    $$
    The term ${\rm e}^{2\min\{\sqrt{\log_2 x}, 1/(1-q) +\frac{\theta}{2}\log_2 x \}}$ will disappear when raised to the power of $1-q$, in fact $(1-q)\sqrt{\log_2 x} \ll 1$ thus  ${\rm e}^{2\min\{\sqrt{(1-q)\log_2 x}, 1 +\frac{\theta}{2}(1-q)\log_2 x \}} \ll 1$.
    We have then
    $$
    \bigg(\mathbb{E}\bigg[ \bigg(\frac{1}{\delta} \int_{\mathcal{S}} \big|F_x(it +4V/\log x)\big|^2 {\rm d}t\bigg)^2 \bigg]\bigg)^{1-q} \ll \bigg(\frac{\log x}{V}\bigg)^{2(1-q)}.
    $$
The proof of the Rademacher cases is quite similar to the Steinhaus case, with only a slight difference. The Rademacher case is a bit more complicated in terms of notation because it lacks the convenient property of "translation invariance in law" found in the Euler products. With minor modifications, the Rademacher lower bound follows simply by imitating the arguments from the Steinhaus case.
\end{proof}

\section{The case of character sums}
In this section, we will prove the Theorem \ref{theoreme_principal_3}.
\subsection{Preliminaries }
In this subsection, we state some useful lemmas. 
If $W(\chi)$ any function, then we define  $\mathbb{E}^{char}[W]:= \frac{1}{r-1}\sum_{\chi \!\!\!\! \mod r} W(\chi).$
We borrow from Harper this useful following lemma.
\begin{lemma}\label{Approcimation_lemma_harper}
    Let $n \in \mathbb{N}$ be large number, and $\lambda>0$ be small. There exist functions $g:\mathbb{R} \to \mathbb{R}$ and $g_{N+1} : \mathbb{R} \to \mathbb{R}$ which are depend on $\lambda$ such that, if we define $g_j(x):=g(x-j)$ for all integers $|j|\leqslant N$, we have the following properties
    \begin{enumerate}
        \item $\sum_{|j|\leqslant N}g_j(x)+g_{N+1}(x)=1$ for all $x\in \mathbb{R}$;
        \item $g(x)\geqslant 0 $ for all $x\in \mathbb{R}$, and $g (x)\leqslant \lambda$ whenever $|x|>1$;
        \item $g_{N+1}(x)\geqslant 0 $ for all $x\in \mathbb{R}$, and $g_{N+1}(x) \leqslant \lambda$ whenever $|x|\leqslant N$;
        \item for all $\ell \in \mathbb{N}$ for all $x \in \mathbb{R}$, we have the derivative estimate $\big| \frac{{\rm d}^{\ell}}{{\rm d}x^{\ell}} g(x)\big|\leqslant \frac{1}{\pi (\ell +1)}\big( \frac{2 \pi }{\lambda}\big)^{\ell +1}$.
    \end{enumerate}
\end{lemma}
\begin{proof}
    See the proof of \cite[Approximation Result 1]{Harper4}.
\end{proof}
\noindent The following lemma, attributed to Harper, serves as the link connecting characters and random multiplicative functions.
\begin{lemma}\label{proposition-harper_approx_character}
    Let the function $g_j$, with associated parameters $N, \lambda,$ be as in Lemma \ref{Approcimation_lemma_harper}. Suppose that $x\geqslant 1$, and let $A$ a subset of integers numbers in $[1,x+y]$, and let $(c(n))_{n\in A}$ be any complex number having absolute values $\leqslant 1$. Furthermore, let $P$ be large, and let $Y\in \mathbb{N}$ be such that $(x+y)P^{400(Y/\lambda)^2\log(N\log P)}<r$. Let $f$ denote a Steinhaus multiplicative function.
    Then for any indices $-N \leqslant j(1),j(2), ... ,j(Y)\leqslant N+1$, and for any sequences $(a_1(p))_{p \leqslant P}$, $(a_1(p^2))_{p \leqslant P}$,..., $(a_Y(p))_{p \leqslant P}$, $(a_Y(p^2))_{p \leqslant P}$ of complex numbers having absolute values $\leqslant 1$, we have
    $$
    \begin{aligned}
            & \mathbb{E}^{char} \Bigg[\prod_{i=1}^{Y}g_{j(i)}\bigg( \mathcal{R} e\bigg( \sum_{p\leqslant P} \frac{a_i(p) \chi (p) }{\sqrt{p}}+\frac{a_i(p^2) \chi (p^2) }{p}\bigg) \bigg)\bigg|\sum_{n \in A} c(n)\chi(n)\bigg|^2\Bigg]
            \\ &= \mathbb{E}\Bigg[\prod_{i=1}^{Y}g_{j(i)}\bigg( \mathcal{R}e\bigg( \sum_{p\leqslant P} \frac{a_i(p)f(p) }{\sqrt{p}}+\frac{a_i(p^2) f (p^2) }{p}\bigg) \bigg)\bigg|\sum_{n \in A} c(n)f(n)\bigg|^2\Bigg] + O\bigg(\frac{|A|}{(N\log P)^{Y(1/\lambda)^2}}\bigg).
    \end{aligned}
    $$
\end{lemma}
\begin{proof}
    This Lemma is a direct consequence from \cite[Proposition 1]{Harper_charac}.
\end{proof}
\subsection{Proof of Theorem \ref{theoreme_principal_3}.}
Let $x+y\leqslant r$. We define $L=\min\{x+y,r/(x+y) \}+3$ and $P=\exp((\log L)^{1/10})$. We set $M:= \lfloor(\log P)^{1.02}\rfloor$ and $N:= \log_2 P$. For each $|k|\leqslant M$ we define 
\begin{equation}\label{quantity_1}
    S_k(\chi):= \mathcal{R }\bigg( \sum_{p\leqslant P} \frac{\chi (p)}{p^{\frac{1}{2}+\frac{ik}{(\log P)^{1.01}}}} + \frac{\chi (p^2)}{p^{1+i\frac{2ik}{(\log P)^{1.01}}}}\bigg).
\end{equation}
When transitioning to the random multiplicative function, $S_k(\chi) $ can be approximately expressed as $\exp\big(F_P(it)\big)$. This expression encapsulates sufficient information essential for controlling the sum in the form of $\sum_{\substack{ d \leqslant x+y }}^{P,x} \big| M^{+}\big(\frac{x}{d},\frac{y}{d},P;f\big) \big|^2$.
\subsubsection{The conditioning argument}\label{subsection_condition_argument}
 We start first with some simplifications. By H\"{o}lder's inequality and by applying \cite[theorem 5.1]{Hiledbrand_tenenbaum}, we have
$$
\begin{aligned}
    \mathbb{E}^{char}\bigg[\big|M^{+}(x,y,x^{\frac{1}{\log_2 x}};\chi)  \big|^{2q} \bigg]& \leqslant \Bigg( \mathbb{E}^{char} \bigg[\big|M^{+}(x,y,x^{\frac{1}{\log_2 x}};\chi)  \big|^2\bigg] \Bigg)^{q}
    \\ & \leqslant \Bigg( \sum_{\substack{x<n\leqslant x+y\\ P(n) \leqslant x^{\frac{1}{\log_2 x}}}} 1 \Bigg)^q \ll \bigg(\frac{y}{\log x}\bigg)^q.
\end{aligned}
$$
Following the step as Lemma \ref{p_n_less_x_power_-(K+1)}, we prove easily that
$$
\mathbb{E}^{char}\bigg[\big| M^{-}(x,y,x;\chi) \big|^{2q} \bigg] \ll \bigg(\frac{y}{\log x}\bigg)^q.
$$
Thus $\mathbb{E}^{char}\bigg[\big|M^{+}(x,y,x^{\frac{1}{\log_2 x}};\chi)  \big|^{2q}\bigg] $ and $ \mathbb{E}^{char}\bigg[\big| M^{-}(x,y,x;\chi) \big|^{2q}\bigg]$ 
which can be neglected compared to 
\begin{equation}\label{dominant_term}
    y \min\bigg\{ 1, \theta \sqrt{\log_2 x} + \frac{1}{(1-q)\sqrt{\log_2 L}}\bigg\}.
\end{equation}
We will focus on studying $ \sum_{\substack{x<n\leqslant x+y\\ x^{\frac{1}{\log_2 x}} < P(n) \leqslant x }} \chi(n)$.
Following Harper \cite{Harper_charac}, we define $\sigma(\textbf{j}):= \mathbb{E}^{char} \big[ \sum_{i=-M}^{M} g_{j_i}(S_i(\chi)) \big]$ and $$\mathbb{E}^{\textbf{j}}\big[W\big]:= \frac{1}{\sigma(\textbf{j})} \mathbb{E}^{char} \big[ W \sum_{i=-M}^{M} g_{j_i}(S_i(\chi))\big]$$ for all function $W(\chi)$. Recall from Lemma \ref{Approcimation_lemma_harper} that for every $i$, $ \sum_{j=-N}^{N} g_j(S_i(\chi))=1$. We have
$$
\begin{aligned}
    & \mathbb{E}^{char}\Bigg[\bigg| \sum_{\substack{x<n\leqslant x+y\\ x^{\frac{1}{\log_2 x}} < P(n) \leqslant x }} \chi(n) \bigg|^{2q} \Bigg]  = \mathbb{E}^{char} \Bigg[\prod_{i=-M}^{M} \bigg( \sum_{j=-N}^{N} g_j(S_i(\chi)\bigg) \bigg| \sum_{\substack{x<n\leqslant x+y\\ x^{\frac{1}{\log_2 x}} < P(n) \leqslant x }} \chi(n) \bigg|^{2q} \Bigg]
    \\ & = \sum_{-N \leqslant j_{-M},...,j_0,...,j_M \leqslant N+1} \mathbb{E}^{char} \Bigg[ \prod_{i=-M}^{M} g_{j_i}(S_i(\chi)) \bigg| \sum_{\substack{x<n\leqslant x+y\\ x^{\frac{1}{\log_2 x}} < P(n) \leqslant x }} \chi(n) \bigg|^{2q} \Bigg]
    \\ &= \sum_{\textbf{j}} \sigma(\textbf{j})\mathbb{E}^{\textbf{j}}\Bigg[ \bigg| \sum_{\substack{x<n\leqslant x+y\\ x^{\frac{1}{\log_2 x}} < P(n) \leqslant x }} \chi(n) \bigg|^{2q} \Bigg]
\end{aligned}
$$
where $\textbf{j}$ is the $2M+1$-tuple $ (j_{-M},...,j_0,...,j_M)$ such that $-N \leqslant j_{-M},...,j_0,...,j_M \leqslant N+1$. 
We have by H\"{o}lder's inequality 
$$
\mathbb{E}^{char}\Bigg[\bigg| \sum_{\substack{x<n\leqslant x+y\\ x^{\frac{1}{\log_2 x}} < P(n) \leqslant x }} \chi(n) \bigg|^{2q} \Bigg] \leqslant \sum_{\textbf{j}} \sigma(\textbf{j})\Bigg(\mathbb{E}^{\textbf{j}}\Bigg[ \bigg| \sum_{\substack{x<n\leqslant x+y\\ x^{\frac{1}{\log_2 x}} < P(n) \leqslant x }} \chi(n) \bigg|^{2} \Bigg]\Bigg)^q.
$$
The quantity $\mathbb{E}^{\textbf{j}}$ represents our character sum version of a conditional expectation. Specifically, when considering the constant function 1, as defined, $\mathbb{E}^{\textbf{j}}[1]=1$ for all choices of the vector $\textbf{j}$. Thus, all the information we need is encapsulated in $\mathbb{E}^{\textbf{j}}$ itself.
\subsubsection{Passing to the random multiplicative functions}
In the following section, we will establish a connection between characters and random multiplicative functions, enabling us to leverage some properties associated to multiplicative chaos.
Let $f$ be a Steinhaus random multiplicative function. We define, for all random variable $W$, the following expectation $$\mathbb{E}^{\textbf{j}, rand}[W]:= \frac{1}{\sigma^{rand}(\textbf{j})} \mathbb{E}\bigg[W \prod_{i=-M}^M g_{j_i}(S_i(f)) \bigg]$$ 
where 
\begin{equation}\label{sigma_j_rand}
    \sigma^{rand}(\textbf{j}):= \mathbb{E}\bigg[ \prod_{i=-M}^M g_{j_{i}}\big(S_i(f)\big)\bigg]
\end{equation}
(the analogue $\sigma(\textbf{j})$ define earlier in the subsection \ref{subsection_condition_argument}). By Lemma \ref{proposition-harper_approx_character}, with $Y=2M+1$, and for $N $ and $\lambda $ such that 
$ (x+y)P^{400\big((2M+1)/\lambda\big)^2\log(N\log P)}<r$, we get 
$$
\begin{aligned}
    \mathbb{E}^{\textbf{j}}\bigg[\bigg| \sum_{\substack{x<n\leqslant x+y\\ x^{\frac{1}{\log_2 x}} < P(n) \leqslant x }} \chi(n) \bigg|^{2} \bigg]= & \frac{\sigma^{rand}(\textbf{j})}{\sigma(\textbf{j})} 
 \mathbb{E}^{\textbf{j}, rand} \bigg[  \bigg| \sum_{\substack{x<n\leqslant x+y\\ x^{\frac{1}{\log_2 x}} < P(n) \leqslant x }} f(n) \bigg|^2\bigg] \\ & +O\bigg( \frac{1}{\sigma(\textbf{j})}\frac{y}{(N\log P)^{(2M+1)/\lambda^2}}\bigg).
\end{aligned}
$$
For the big Oh term, when we multiply with $\sigma(\textbf{j})$ followed by a sum over $\textbf{j}$, we get 
$$
\sum_{\textbf{j}} \frac{y}{(N\log P)^{(2M+1)/\lambda^2}} =\frac{y(2N+1)^{2M+1}}{(N\log P)^{(2M+1)/\lambda^2}}  \ll \frac{y}{\log P}
$$
which can be neglected compared to 
\eqref{dominant_term}.
By using Lemma \ref{Approcimation_lemma_harper} in the case where $\chi=1$ (see also \cite[Proposition 2]{Harper_charac}), 
we have 
$\sigma^{rand}(\textbf{j})\ll \sigma(\textbf{j}) +  \frac{1}{(N\log P)^{(2M+1)/\lambda^2}} $. Therefore we have
$$
\begin{aligned}
    & \sum_{\textbf{j}} \sigma(\textbf{j})\bigg( \frac{\sigma ^{rand}(\textbf{j})}{\sigma(\textbf{j})} \mathbb{E}^{\textbf{j}, rand}  \bigg[   \bigg| \sum_{\substack{x<n\leqslant x+y\\ x^{\frac{1}{\log_2 x}} < P(n) \leqslant x }} f(n) \bigg|^2 \bigg] \bigg)^q
    \\ & \ll \sum_{\textbf{j}} \sigma^{rand}(\textbf{j})\bigg(  \mathbb{E}^{\textbf{j}, rand} \bigg[   \bigg| \sum_{\substack{x<n\leqslant x+y\\ x^{\frac{1}{\log_2 x}} < P(n) \leqslant x }} f(n) \bigg|^2 \bigg] \bigg)^q 
    \\ & + \bigg( \frac{1}{(N\log P)^{(2M+1)/\lambda^2}} \bigg)^{1-q} \sum_{\textbf{j}} \bigg(  \mathbb{E} \bigg[ \prod_{i=-M}^M g_{j_i} (S_i(f))  \bigg| \sum_{\substack{x<n\leqslant x+y\\ x^{\frac{1}{\log_2 x}} < P(n) \leqslant x }} f(n) \bigg|^2 \bigg] \bigg)^q.
\end{aligned}
$$
\noindent We finish this subsection by proving that the last term above is also negligible compared to \eqref{dominant_term}.
By H\"{o}lder's inequality to the sum $\textbf{j}$, we have the last term above is
$$
\begin{aligned}
    & \bigg( \frac{1}{(N\log P)^{(2M+1)/\lambda^2}} \bigg)^{1-q} ((2N+2)^{2M+1})^{1-q} \bigg( \sum_{\textbf{j}} \mathbb{E}^{\textbf{j}, rand} \bigg[  \bigg| \sum_{\substack{x<n\leqslant x+y\\ x^{\frac{1}{\log_2 x}} < P(n) \leqslant x }} f(n) \bigg|^2 \bigg] \bigg)^q 
    \\ & = \bigg( \frac{(2N+2)^{2M+1}}{(N\log P)^{(2M+1)/\lambda^2}} \bigg)^{1-q}  \bigg(  \mathbb{E} \bigg[  \bigg| \sum_{\substack{x<n\leqslant x+y\\ x^{\frac{1}{\log_2 x}} < P(n) \leqslant x }} f(n) \bigg|^2 \bigg] \bigg)^q 
    \\ & \ll y^q \bigg( \frac{(2N+2)^{2M+1}}{(N\log P)^{(2M+1)/\lambda^2}} \bigg)^{1-q}\ll y^q{\rm e}^{-(1-q)\log_2 P}
\end{aligned}
$$
which can be neglected compared to \eqref{dominant_term}.
From now on, the problem of characters is probabilistic problem. In other words, it suffices to prove
$$
\sum_{\textbf{j}} \sigma^{rand}(\textbf{j}) \bigg( \mathbb{E}^{\textbf{j}, rand}\bigg[ \bigg| \!\!\!\!\!\!\!\!\sum_{\substack{x<n\leqslant x+y\\ x^{\frac{1}{\log_2 x}} < P(n) \leqslant x }} \!\!\!\!\!\!\!\!f(n) \bigg|^2 \bigg]\bigg)^q \ll \bigg( y \min\bigg\{ 1, \theta \sqrt{\log_2 P} + \frac{1}{(1-q)\sqrt{\log_2 P}}\bigg\} \bigg)^q.
$$
\subsubsection{Euler product}
In this section, we will convert the expectation $\mathbb{E}^{\textbf{j}, rand}\bigg[ \bigg| \sum_{\substack{x<n\leqslant x+y\\ x^{\frac{1}{\log_2 x}} < P(n) \leqslant x }} f(n) \bigg|^2 \bigg] $ in form of Euler product.
This conversion will enable us to apply multiplicative chaos techniques.
Recall the expectation conditioning on $(f(p))_{p\leqslant P}$ is given by $\mathbb{E}_P [.]= \mathbb{E}\big[\,\,\,.\,\,\,|\, (f(p))_{p\leqslant P}\big]$. Since the product $\prod_{i=-M}^M g_{j_i}(S_i(f))$ in the definition of $\mathbb{E}^{\textbf{j},rand}$ depends only on $(f(p))_{p\leqslant P}$, thus we have 
$$
\begin{aligned}
   \mathbb{E}^{\textbf{j}, rand}\bigg[ \bigg| \sum_{\substack{x<n\leqslant x+y\\ x^{\frac{1}{\log_2 x}} < P(n) \leqslant x }} f(n) \bigg|^2 \bigg] & = \frac{1}{\sigma^{rand}(\textbf{j})} \mathbb{E}\bigg[ \mathbb{E}_{P}\bigg[ \prod_{i=-M}^M g_{j_i}(S_i(f)) \bigg| \sum_{\substack{x<n\leqslant x+y\\ x^{\frac{1}{\log_2 x}} < P(n) \leqslant x }} f(n) \bigg|^2 \bigg] \bigg] 
   \\ &= \frac{1}{\sigma^{rand}(\textbf{j})}\mathbb{E}\bigg[  \prod_{i=-M}^M g_{j_i}(S_i(f))\!\!\!\!\!\!\!\! \sum_{\substack{ d  \leqslant x+y \\ P(d)> x^{\frac{1}{\log_2 x}} }}^{P,x}\!\!\!\!\!\bigg| M^{+}\bigg(\frac{x}{d},\frac{y}{d},P;f\bigg) \bigg|^2 \bigg]
   \\ & = \mathbb{E}^{\textbf{j}, rand}\bigg[  \sum_{\substack{ d  \leqslant x+y \\ P(d)> x^{\frac{1}{\log_2 x}} }}^{P,x} \bigg| M^{+}\bigg(\frac{x}{d},\frac{y}{d},P;f\bigg) \bigg|^2 \bigg].
\end{aligned}
$$
Recall that $X={\rm e}^{\sqrt{x}}$. For any $\textbf{j}$, we have
$$
\begin{aligned}
     &   \mathbb{E}^{\textbf{j}, rand}\bigg[  \sum_{\substack{ d  \leqslant x+y \\ P(d)> x^{\frac{1}{\log_2 x}} }}^{P,x}\bigg| M^{+}\bigg(\frac{x}{d},\frac{y}{d},P;f\bigg) \bigg|^2 \bigg]
    \\ &\ll   \mathbb{E}^{\textbf{j}, rand}\bigg[  \sum_{\substack{ d  \leqslant x+y \\ P(d)> x^{\frac{1}{\log_2 x}} }}^{P,x} \frac{X}{d} \int_{d}^{d(1+1/X)}\bigg| M^{+}\bigg(\frac{x}{t},\frac{y}{t},P;f\bigg) \bigg|^2   {\rm d}t\bigg] 
    \\ &  +\mathbb{E}^{\textbf{j}, rand}\bigg[  \sum_{\substack{ d  \leqslant x+y \\ P(d)> x^{\frac{1}{\log_2 x}} }}^{P,x} \frac{X}{d} \int_{d}^{d(1+1/X)}\bigg| \sum_{\substack{\frac{x+y}{t}<n\leqslant \frac{x+y}{d}\\ P(n) \leqslant P \\ }} f(n) \bigg|^2+ \bigg| \sum_{\substack{\frac{x}{t}<n\leqslant \frac{x}{d}\\ P(n) \leqslant P \\ }} f(n) \bigg|^2  {\rm d}t   \bigg].
\end{aligned}
$$
We will deal first with the last term.
Since $\sum_{\textbf{j}}\sigma^{rand}(\textbf{j})=1$, we have by H\"{o}lder's inequality and by following the same techniques as in the proof of inequality \eqref{inequality_0101}, we have
$$
\begin{aligned}
    & \sum_{\textbf{j}}\sigma^{rand}(\textbf{j})\bigg( \mathbb{E}^{\textbf{j}, rand}\bigg[  \sum_{\substack{ d  \leqslant x+y \\ P(d)> x^{\frac{1}{\log_2 x}} }}^{P,x} \frac{X}{d} \int_{d}^{d(1+1/X)}\bigg| \sum_{\substack{\frac{x+y}{t}<n\leqslant \frac{x+y}{d}\\ P(n) \leqslant P \\ }} f(n) \bigg|^2 {\rm d}t \bigg]  \bigg)^q 
    \\& \leqslant \bigg(\sum_{\textbf{j}}\sigma^{rand}(\textbf{j})\mathbb{E}^{\textbf{j}, rand}\bigg[  \sum_{\substack{ d  \leqslant x+y \\ P(d)> x^{\frac{1}{\log_2 x}} }}^{P,x} \frac{X}{d} \int_{d}^{d(1+1/X)}\bigg| \sum_{\substack{\frac{x+y}{t}<n\leqslant \frac{x+y}{d}\\ P(n) \leqslant P \\ }} f(n) \bigg|^2 {\rm d}t \bigg]  \bigg)^q 
    \\ & = \bigg( \mathbb{E}\bigg[  \sum_{\substack{ d  \leqslant x+y \\ P(d)> x^{\frac{1}{\log_2 x}} }}^{P,x} \frac{X}{d} \int_{d}^{d(1+1/X)}\bigg| \sum_{\substack{\frac{x+y}{t}<n\leqslant \frac{x+y}{d}\\ P(n) \leqslant P \\ }} f(n) \bigg|^2 {\rm d}t \bigg] \bigg)^q.
    \end{aligned}
$$
Now by using the fact that $$\mathbb{E}\bigg[  \bigg| \sum_{\substack{\frac{x+y}{t}<n\leqslant \frac{x+y}{d}\\ P(n) \leqslant P \\ }} f(n) \bigg|^2 \bigg]\leqslant \sum_{\substack{\frac{x+y}{t}<n\leqslant \frac{x+y}{d}\\ P(n) \leqslant P \\ }} 1,$$ we get the above expression is 
$$
\begin{aligned}
    \\ & \leqslant \bigg(  \sum_{\substack{ d  \leqslant x+y \\ P(d)> x^{\frac{1}{\log_2 x}} }}^{P,x} \frac{X}{d} \int_{d}^{d(1+1/X)} \sum_{\substack{\frac{x+y}{d(1+1/X)}<n\leqslant \frac{x+y}{d}\\ P(n) \leqslant P  }} 1 {\rm d}t  \bigg)^q
    \\ & \leqslant  \bigg(    \sum_{\substack{\frac{x+y}{(1+1/X)}<n\leqslant x+y }} 1   \bigg)^q \ll \bigg(    \frac{x}{X}  \bigg)^q \ll \bigg(  \frac{y}{\sqrt{\log_2 x}}   \bigg)^q .
\end{aligned}
$$
We prove similarly that 
$$
\sum_{\textbf{j}}\sigma^{rand}(\textbf{j})\bigg( \mathbb{E}^{\textbf{j}, rand}\bigg[  \sum_{\substack{ d  \leqslant x+y \\ P(d)> x^{\frac{1}{\log_2 x}} }}^{P,x} \frac{X}{d} \int_{d}^{d(1+1/X)}\bigg| \sum_{\substack{\frac{x}{t}<n\leqslant \frac{x}{d}\\ P(n) \leqslant P \\ }} f(n) \bigg|^2 {\rm d}t \bigg]  \bigg)^q \ll \bigg(  \frac{y}{\sqrt{\log_2 x}}   \bigg)^q.
$$
which is more than enough for us.
Now let study
$$ \mathbb{E}^{\textbf{j}, rand}\bigg[  \sum_{\substack{ d  \leqslant x+y \\ P(d)> x^{\frac{1}{\log_2 x}} }}^{P,x} \frac{X}{d} \int_{d}^{d(1+1/X)}\bigg| M^{+}\bigg(\frac{x}{t},\frac{y}{t},P;f\bigg) \bigg|^2 {\rm d}t \bigg]  .$$
By using \eqref{Harper_sieve} followed with a change of variable $z=x/t$, we have
$$
\begin{aligned}
   &  \!\!\!\!\!\!\!\!\!\!\!\! \!\!\!\!\!\!\!\!\!\!\!\!\mathbb{E}^{\textbf{j}, rand}\bigg[  \sum_{\substack{ d  \leqslant x+y \\ P(d)> x^{\frac{1}{\log_2 x}} }}^{P,x} \frac{X}{d} \int_{d}^{d(1+1/X)}\bigg| M^{+}\bigg(\frac{x}{t},\frac{y}{t},P;f\bigg) \bigg|^2   {\rm d}t  \bigg] 
    \\ &  \!\!\!\!\!\!\!\!\!\!\!\!\!\!\!\!\!\!\!\!\!\!\!\!\leqslant  \mathbb{E}^{\textbf{j}, rand}\bigg[  \int_{x^{1/\log_2 x}}^{(x+y)(1+1/X)}\sum_{\substack{ t/(1+1/X) \leqslant d \leqslant t  }}^{P,x} \frac{X}{d} \bigg| M^{+}\bigg(\frac{x}{t},\frac{y}{t},P;f\bigg) \bigg|^2 {\rm d}t   \bigg] 
    \\ &  \!\!\!\!\!\!\!\!\!\!\!\!\!\!\!\!\!\!\!\!\!\!\!\! \ll  \mathbb{E}^{\textbf{j}, rand}\bigg[ \frac{x}{\log P} \int^{x^{1-1/\log_2 x}}_{\frac{1}{(1+1/\delta)(1+1/X)}}  \big|M^{+}\big(z,z/\delta,P;f\big) \big|^2  \frac{{\rm d}z}{z^2} \bigg] .
\end{aligned}
$$
By applying Parseval's identity (Lemma \ref{parseval}), we have
$$
\begin{aligned}
    & \int^{x^{1-1/\log_2 x}}_{\frac{1}{(1+1/\delta)(1+1/X)}}  \big|M^{+}\big(z,z/\delta,P;f\big) \big|^2 \ll \int_{-\infty}^{+ \infty} \frac{\big| (1+1/\delta)^{1/2+it} -1 \big|^2}{\big|1/2 +it  \big|^2} \big|F_P(it ) \big|^2{\rm d}t.  
\end{aligned}
$$
By using the same decomposition as in Lemma \ref{sum_over_K}, we have $ | (1+1/\delta)^{1/2+it} -1| \ll \frac{|1/2+it|}{\delta}$ for $|t|\leqslant \delta $. For $|t| > \delta $, we give the trivial bound $ | (1+1/\delta)^{1/2+it} -1| \ll 1$. We have then
$$
\begin{aligned}
    & \int_{-\infty}^{+ \infty} \frac{\big| (1+1/\delta)^{1/2+it} -1 \big|^2}{\big|1/2 +it  \big|^2} \big|F_P(it ) \big|^2{\rm d}t
    \\ & \ll \frac{1}{\delta^2}\int_{-\delta}^{\delta}  \big|F_P(it ) \big|^2{\rm d}t + \int_{\delta\leqslant |t|\leqslant \delta^5} \frac{\big|F_P(it ) \big|^2}{|1/2+it|}{\rm d}t
    + \int_{|t|> \delta^5} \frac{\big|F_P(it ) \big|^2}{|1/2+it|}{\rm d}t.
\end{aligned}
$$
for the second and last term, we continue with some additive simplifications
$$
\int_{\delta\leqslant |t|\leqslant \delta^5} \frac{\big|F_P(it ) \big|^2}{|1/2+it|}{\rm d}t\ll \sum_{\delta <|n| \leqslant \delta^5} \frac{1}{n^2} \big|F_P(it+in ) \big|^2 {\rm d}t
$$
and
$$
\begin{aligned}
    \int_{|t|> \delta^5} \frac{\big|F_P(it ) \big|^2}{|1/2+it|}{\rm d}t &\ll \sum_{|n|>\delta^5}\frac{1}{n^2}\int_{-1/2}^{1/2} \big|F_P(it+in ) \big|^2 {\rm d}t 
\\ & \ll \frac{1}{\delta} \sum_{n \in \mathbb{N}} \frac{1}{(|n|+1)^{8/5}}\int_{-1/2}^{1/2} \big|F_P(it+in ) \big|^2 {\rm d}t.
\end{aligned}
$$
For the sake of simplicity, we set as in section \ref{Intermediate upper bound}
$$
B_1(\delta,P;f):= \sum_{\textbf{j}}\sigma^{rand}(\textbf{j})\bigg( \frac{y}{\log P}\mathbb{E}^{\textbf{j}, rand}\bigg[   \frac{1}{\delta}\int_{-\delta}^{\delta}  \big|F_P(it ) \big|^2{\rm d}t  \bigg] \bigg)^q,
$$
$$
B_2(\delta,P;f):= \sum_{\textbf{j}}\sigma^{rand}(\textbf{j})\bigg( \frac{x}{\log P}\mathbb{E}^{\textbf{j}, rand}\bigg[   \sum_{\delta <|n| \leqslant \delta^5} \frac{1}{n^2} \int_{-1/2}^{1/2} \big|F_P(it+in ) \big|^2 {\rm d}t \bigg] \bigg)^q,
$$
and
$$
B_3(\delta,P;f):=  \sum_{n \in \mathbb{N}} \frac{1}{(|n|+1)^{16/15}}\sum_{\textbf{j}}\sigma^{rand}(\textbf{j})\bigg( \frac{y}{\log P}\mathbb{E}^{\textbf{j}, rand}\bigg[  \int_{-1/2}^{1/2} \big|F_P(it+in ) \big|^2 {\rm d}t \bigg] \bigg)^q.
$$
One can see easily that
$$\begin{aligned}
    & \sum_{\textbf{j}}\sigma^{rand}(\textbf{j})\bigg( \mathbb{E}^{\textbf{j}, rand}\bigg[ \frac{x}{\log P} \int^{x^{1-1/\log_2 x}}_{\frac{1}{(1+1/\delta)(1+1/X)}}  \bigg| \sum_{\substack{z<n\leqslant z(1+1/\delta)\\ P(n) \leqslant P \\ }} f(n) \bigg|^2  \frac{{\rm d}z}{z^2} \bigg]  \bigg)^q 
    \\& \ll B_1(\delta,P;f)+B_2(\delta,P;f)+B_3(\delta,P).
\end{aligned}
$$
    By using \cite[inequality (3.5)]{Harper_charac}, we have
    \begin{equation}\label{bound_b_3}
        B_3(\delta,P;f) \ll \bigg(\frac{y}{1+(1-q)\sqrt{\log_2 P}}\bigg)^q.
    \end{equation}
    Now the problem is reduced to prove uniformly for any $ 2/3 \leqslant q\leqslant 1$ and any $x+y\leqslant r$, we have
    $$
    B_1(\delta,P;f) + B_2(\delta,P;f) \ll \bigg(  y \min\bigg\{ 1, \theta \sqrt{\log_2 P} + \frac{1}{(1-q)\sqrt{\log_2 P}}\bigg\} \bigg)^q.
    $$
\subsubsection{Multiplicative Chaos and conclusion}
In this section, we will apply the multiplicative chaos result of previous sections.
We start by some adjustment to $ B_1(\delta,P;f)$ and  $B_2(\delta,P;f)$. Following Harper \cite{Harper_charac}, we need first to move from the integral to discretised version. For the sake of readability, we let $K_P:= (\log P)^{1.01}$ We have
$$
\begin{aligned}
    &B_1(\delta,P;f) \ll  \sum_{\textbf{j}}\sigma^{rand}(\textbf{j})\bigg( \frac{y}{\delta\log P}\mathbb{E}^{\textbf{j}, rand}\bigg[ \sum_{|k|\leqslant \delta K_P } \!\!\!\!  \int_{-\frac{1}{2 K_P }}^{\frac{1}{2 K_P }}  \bigg|F_P\bigg(\frac{ik}{ K_P }+it \bigg) 
     \!-\! F_P\bigg(\frac{ik}{ K_P }\bigg)\bigg|^2{\rm d}t  \bigg] \bigg)^q 
    \\&+ \sum_{\textbf{j}}\sigma^{rand}(\textbf{j})\bigg( \frac{y}{\delta\log P}\mathbb{E}^{\textbf{j}, rand}\bigg[ \sum_{|k|\leqslant \delta K_P }   \int_{-\frac{1}{2 K_P }}^{\frac{1}{2 K_P }}  \bigg| F_P\bigg(\frac{ik}{ K_P }\bigg)\bigg|^2{\rm d}t  \bigg] \bigg)^q .
\end{aligned}
$$
By applying H\"{o}lder's inequality applied to sum over $\textbf{j}$ once again, we have
$$
\begin{aligned}
    & \sum_{\textbf{j}}\sigma^{rand}(\textbf{j})\bigg( \frac{y}{\log P}\mathbb{E}^{\textbf{j}, rand}\bigg[ \sum_{|k|\leqslant \delta K_P }   \frac{1}{\delta}\int_{-\frac{1}{2 K_P }}^{\frac{1}{2 K_P }}  \bigg|F_P\bigg(\frac{ik}{ K_P }+it \bigg) 
     - F_P\bigg(\frac{ik}{ K_P }\bigg)\bigg|^2{\rm d}t  \bigg] \bigg)^q 
    \\ & \leqslant \bigg(\frac{y}{\log P}  \bigg)^q \bigg( \sum_{|k|\leqslant \delta K_P }   \frac{1}{\delta}\int_{-\frac{1}{2 K_P }}^{\frac{1}{2 K_P }} \mathbb{E}\bigg[  \bigg|F_P\bigg(\frac{ik}{ K_P }+it \bigg)
     - F_P\bigg(\frac{ik}{ K_P }\bigg)\bigg|^2 \bigg] {\rm d}t  \bigg)^q .
\end{aligned}
$$
We have by Lemma \ref{perturabtion}
$$
\begin{aligned}
    & \sum_{|k|\leqslant \delta K_P }   \frac{1}{\delta}\int_{-\frac{1}{2 K_P }}^{\frac{1}{2 K_P }}\mathbb{E}\bigg[  \bigg|F_P\bigg(\frac{ik}{ K_P }+it \bigg) - F_P\bigg(\frac{ik}{ K_P }\bigg)\bigg|^2  \bigg] 
    \ll  (\log P)^{0.99} .
\end{aligned}
$$
Thus
$$
\begin{aligned}
    B_1(\delta,P;f) \ll &  \sum_{\textbf{j}}\sigma^{rand}(\textbf{j})\bigg( \frac{y}{\log P}\mathbb{E}^{\textbf{j}, rand}\bigg[ \sum_{|k|\leqslant \delta K_P }   \frac{1}{\delta}\frac{1}{ K_P }  \bigg| F_P\bigg(\frac{ik}{ K_P }\bigg)\bigg|^2  \bigg] \bigg)^q
     \\ & + \frac{y^q}{(\log P)^{0.01q}}.
\end{aligned}
$$
Following the same steps, we have for $B_2(\delta,P;f) $
$$
\begin{aligned}
     B_2(\delta,P;f) \ll&  \sum_{\textbf{j}}\sigma^{rand}(\textbf{j})\bigg( \frac{y}{\log P}\mathbb{E}^{\textbf{j}, rand}\bigg[ \sum_{  \delta K_P <|k|\leqslant \delta^5 K_P  }  \frac{1}{n^2}\frac{1}{ K_P }  \bigg| F_P\bigg(\frac{ik}{ K_P }\bigg)\bigg|^2 \bigg] \bigg)^q
     \\ & + \frac{y^q}{(\log P)^{0.01q}}.
\end{aligned}
$$
Let define
$$
\widetilde{B}_1(\delta,P;f):= \sum_{\textbf{j}}\sigma^{rand}(\textbf{j})\bigg( \frac{y}{\log P}\mathbb{E}^{\textbf{j}, rand}\bigg[ \sum_{|k|\leqslant \delta K_P }   \frac{1}{\delta}\frac{1}{ K_P }  \bigg| F_P\bigg(\frac{ik}{ K_P }\bigg)\bigg|^2  \bigg] \bigg)^q
$$
and
$$
\begin{aligned}
    \widetilde{B}_2(\delta,P;f):= \sum_{\textbf{j}}\sigma^{rand}(\textbf{j})\bigg( \frac{y}{\log P}\mathbb{E}^{\textbf{j}, rand}\bigg[ \sum_{  \delta K_P <|k|\leqslant \delta^{5}  K_P } &  \frac{1}{n^2}\frac{1}{ K_P }  \bigg| F_P\bigg(\frac{ik}{ K_P }\bigg)\bigg|^2  \bigg] \bigg)^q.
\end{aligned}
$$
The goal for the rest of the paper is to prove the following bounds
\begin{equation}\label{B_1}
      \widetilde{B}_1(\delta,P;f)  \ll \bigg(  y \min\bigg\{ 1, \theta \sqrt{\log_2 P} + \frac{1}{(1-q)\sqrt{\log_2 P}}\bigg\} \bigg)^q
\end{equation}
and 
\begin{equation}\label{B_2}
      \widetilde{B}_2(\delta,P;f)  \ll \bigg(  y \min\bigg\{ 1, \theta \sqrt{\log_2 P} + \frac{1}{(1-q)\sqrt{\log_2 P}}\bigg\} \bigg)^q.
\end{equation}
We will focus on proving \eqref{B_1}, for \eqref{B_2} can be proven similarly. Following Harper \cite[section 3.5]{Harper_charac}, we define the following set 
$$
\mathcal{T}:=\bigg\{ k\in \mathbb{Z}: \bigg| F_P\bigg(\frac{ik}{ K_P }\bigg)\bigg| >(\log P)^{1.1} \text{ or } \bigg| F_P\bigg(\frac{ik}{ K_P }\bigg)\bigg| \leqslant\frac{1}{(\log P)^{1.1}} \bigg\}.
$$
By dividing the sum in $\widetilde{B}_1(\delta,P;f)$ over $k$ to $k\in \mathcal{T}$, $k \notin \mathcal{T}$, $|S_k(f)-j_k|>1$, and $|S_k(f)-j_k|\leqslant1$ and by H\"{o}lder's inequality for the sum over $k$ restricted to $k \in \mathcal{T} $, we have
$$
 \widetilde{B}_1(\delta,P;f) \leqslant  \widetilde{B}_1^{(a)}(\delta,P;f)+ \widetilde{B}_1^{(b)}(\delta,P;f)+\widetilde{B}_1^{(c)}(\delta,P;f)
$$
Where
$$
\widetilde{B}_1^{(a)}(\delta,P;f) := \sum_{\textbf{j}}\sigma^{rand}(\textbf{j})\bigg( \frac{y}{\delta(\log P)^{2.01}}\mathbb{E}^{\textbf{j}, rand}\bigg[ \sum_{\substack{|k|\leqslant \delta K_P \\ k \notin \mathcal{T}} } \mathbb{1}_{|S_k(f)-j_k|\leqslant 1} \bigg| F_P\bigg(\frac{ik}{ K_P }\bigg)\bigg|^2  \bigg] \bigg)^q,
$$
$$
\widetilde{B}_1^{(b)}(\delta,P;f) := \sum_{\textbf{j}}\sigma^{rand}(\textbf{j})\bigg( \frac{y}{\delta(\log P)^{2.01}}\mathbb{E}^{\textbf{j}, rand}\bigg[ \sum_{\substack{|k|\leqslant \delta K_P \\ k \notin \mathcal{T}} } \mathbb{1}_{|S_k(f)-j_k|>1} \bigg| F_P\bigg(\frac{ik}{ K_P }\bigg)\bigg|^2  \bigg] \bigg)^q
$$
and
$$
\widetilde{B}_1^{(c)}(\delta,P;f):=  \bigg(\frac{y}{\log P}\bigg)^q\bigg( \frac{1}{ K_P } \sum_{\substack{|k|\leqslant \delta K_P  }}   \frac{1}{\delta} \mathbb{E}\bigg[  \mathbb{1}_{k \in \mathcal{T}} \bigg| F_P\bigg(\frac{ik}{ K_P }\bigg)\bigg|^2  \bigg] \bigg)^q.
$$
Arguing as Harper \cite[section 3.5]{Harper_charac}, for $\widetilde{B}_1^{(c)}(\delta,P;f)$, we have 
$$
\begin{aligned}
    \mathbb{E}\bigg[  \mathbb{1}_{k \in \mathcal{T}} \bigg| F_P\bigg(\frac{ik}{ K_P }\bigg)\bigg|^2  \bigg] \ll & \big((\log P)^{1.1}\big)^{-0.2}\mathbb{E}\bigg[\bigg| F_P\bigg(\frac{ik}{ K_P }\bigg)\bigg|^{2.2}\bigg] 
     + \bigg(\frac{1}{(\log P)^{1.1}}\bigg)^{2}.
\end{aligned}$$
Using \cite[Euler Product Result 1]{Harper2}, we have $$\mathbb{E}\bigg[\bigg| F_P\bigg(\frac{ik}{ K_P }\bigg)\bigg|^{2.2}\bigg] \ll (\log P)^{1.21}. $$ Thus 
$$
\mathbb{E}\bigg[  \mathbb{1}_{k \in \mathcal{T}} \bigg| F_P\bigg(\frac{ik}{ K_P }\bigg)\bigg|^2  \bigg]\ll \frac{1}{\big((\log P)^{1.1}\big)^{0.2}}\mathbb{E}\bigg[\bigg| F_P\bigg(\frac{ik}{ K_P }\bigg)\bigg|^{2.2}\bigg] \ll (\log P)^{0.99}
$$
which means, the expression $\widetilde{B}_1^{(c)}(\delta,P;f)$ can be neglected. For $\widetilde{B}_1^{(b)}(\delta,P;f)$, by H\"{o}lder's inequality and by rewriting the explicit expression of $\mathbb{E}^{\textbf{j}, rand} $, we have $\widetilde{B}_1^{(b)}(\delta,P;f)$ is less than
$$
\begin{aligned}
\bigg(\frac{y}{\delta(\log P)^{2.01}}\sum_{\substack{|k|\leqslant \delta K_P } } \mathbb{E}\bigg[ \sum_{\textbf{j}} \prod_{i=-M}^{M} g_{j_i}(S_i(f))\mathbb{1}_{k\notin \mathcal{T}}   \mathbb{1}_{|S_k(f)-j_k|>1}\bigg| F_P\bigg(\frac{ik}{ K_P }\bigg)\bigg|^2  \bigg] \bigg)^q.
\end{aligned}
$$
 For $-N\leqslant j_k\leqslant N$ and by using property 2 in Lemma \ref{Approcimation_lemma_harper} we have $ g_{j_k}(S_k(f))\mathbb{1}_{|S_k(f)-j_k|>1}\leqslant~ \lambda$. In addition to that, if $k\notin \mathcal{T}$ then $|S_k(f)|=\log\bigg| F_P\bigg(\frac{ik}{ K_P }\bigg)\bigg|+O(1)\leqslant 1.1\log_2 P +O(1)$. By taking $N\geqslant 1.2\log _2 P$ and for $j_k=N+1$, we have $g_{j_k}(S_k(f))\mathbb{1}_{k\notin \mathcal{T}}\leqslant \lambda$ (by property 3 in Lemma \ref{Approcimation_lemma_harper}). Thus, in any case 
 $$ \prod_{i=-M}^{M} g_{j_i}(S_i(f))\mathbb{1}_{k\notin \mathcal{T}} \mathbb{1}_{|S_k(f)-j_k|>1}\leqslant \lambda \prod_{i\neq k}g_{j_i}(S_i(f)),$$
 and subsequently, upon conducting the summation across all conceivable values of $j_i$ for $i\neq k$ and by using the fact that for every $i$ $ \sum_{-N\leqslant j_i\leqslant N+1}g_{j_i}(S_i(f)) =1$,  we get 
 $$ 
 \begin{aligned}
     \sum_{\textbf{j}} \prod_{i=-M}^{M} g_{j_i}(S_i(f))\mathbb{1}_{k\notin \mathcal{T}} \mathbb{1}_{|S_k(f)-j_k|>1}
     & \leqslant \lambda \sum_{\textbf{j}} \prod_{i\neq k} g_{j_i}(S_i(f)) 
     \\ & = \lambda \sum_{-N \leqslant j_k\leqslant N+1 } 
     \prod_{i\neq k} \sum_{-N\leqslant j_i\leqslant N+1}g_{j_i}(S_i(f)) \\& \ll \lambda N.
 \end{aligned}
 $$ 
 Thus we have
 $$
 \begin{aligned}
  \widetilde{B}_1^{(b)}(\delta,P;f)& \ll  \bigg( \frac{y}{\delta(\log P)^{2.01}}\sum_{\substack{|k|\leqslant \delta K_P } } \lambda N\mathbb{E}\bigg[\bigg| F_P\bigg(\frac{ik}{ K_P }\bigg)\bigg|^2  \bigg] \bigg)^q
 \\ & \ll \bigg(\frac{y}{\log P}\lambda N \sum_{\substack{ n\geqslant 1\\ P(n)\leqslant P}} \frac{1}{n}\bigg)^q \ll \big( y \lambda N \big)^q.
  \end{aligned}
 $$ 
 We take $\lambda \leqslant \frac{1}{N\sqrt{\log_2 P}}$, this gives an acceptable bound.
 The problem now is reduced to prove
 $$
 \begin{aligned}
      \widetilde{B}_1^{(a)}(\delta,P;f) \ll  \bigg(  y \min\bigg\{ 1, \theta \sqrt{\log_2 P} + \frac{1}{(1-q)\sqrt{\log_2 P}}\bigg\} \bigg)^q.
 \end{aligned}
 $$
 Recall that $$\bigg| F_P\bigg(\frac{ik}{ K_P }\bigg)\bigg|=\exp\bigg\{-\mathcal{R}e\sum_{p\leqslant P}  \log\bigg(1-\frac{f(p)}{p^{1/2+\frac{ik}{ K_P }}}\bigg) \bigg\}\asymp \exp\big\{S_{k}(f)\big\}$$ for all $k$. Note that  the only $k$ that contribute to the sum over $k$ are those for which $S_k(f)\in [j(k)-1,j(k)+1]$. Since $k\notin \mathcal{T}$, we have $|S_k(f)|\leqslant 1.1\log_2 P +O(1)$, we have then
 $$
 \begin{aligned}
    \widetilde{B}_1^{(c)}(\delta,P;f) & \ll \sum_{\textbf{j}}\sigma^{rand}(\textbf{j})\bigg( \frac{y}{\delta(\log P)^{2.01}}\mathbb{E}^{\textbf{j}, rand}\bigg[ \!\!\!\!\!\!\!\!\sum_{\substack{|k|\leqslant \delta K_P   \\ |j_k|\leqslant 1.1 \log_2 P +O(1) }} \!\!\!\!\!\!\!\! {\rm e}^{2j_k}  \bigg] \bigg)^q
      \\&= \sum_{\textbf{j}}\sigma^{rand}(\textbf{j})\bigg( \frac{y}{\delta(\log P)^{2.01}} \!\!\!\!\!\!\!\!\sum_{\substack{|k|\leqslant \delta K_P   \\ |j_k|\leqslant 1.1 \log_2 P +O(1) }}\!\!\!\!\!\!\!\!  {\rm e}^{2j_k}\bigg)^q.
 \end{aligned}
 $$
 The last line is because $ \sum_{\substack{|k|\leqslant \delta K_P   \\ |j_k|\leqslant 1.1 \log_2 P +O(1) }}  {\rm e}^{2j_k}$ is deterministic function.  By using the definition of $\sigma^{rand}(\textbf{j})$ in \eqref{sigma_j_rand}, we have
 \begin{equation}\label{sum_over_k_restriction}
 \begin{aligned}
     &\sum_{\textbf{j}}\sigma^{rand}(\textbf{j})\bigg( \frac{y}{\delta(\log P)^{2.01}} \!\!\!\!\!\!\!\!\sum_{\substack{|k|\leqslant \delta K_P   \\ |j_k|\leqslant 1.1 \log_2 P +O(1) }} \!\!\!\!\!\!\!\! {\rm e}^{2j_k} \bigg)^q
     \\ & = \mathbb{E}\Bigg[\sum_{\textbf{j}}\prod_{i=-M}^{M} g_{j_i}(S_i(f))\bigg( \frac{y}{\delta(\log P)^{2.01}}\!\!\!\!\!\!\!\!\sum_{\substack{|k|\leqslant \delta K_P   \\ |j_k|\leqslant 1.1 \log_2 P +O(1) }} \!\!\!\!\!\!\!\! {\rm e}^{2j_k} \bigg)^q\Bigg].
 \end{aligned}
\end{equation}
We divide again the sum above over $k$ to $ |S_k(f)-j_k|\leqslant ~1$ and $ |S_k(f)-j_k|>1$. We deal first with sum over $k$ when 
 $|S_k(f)-j_k|\leqslant ~1$.
Recall that $\sum_{\textbf{j}}\prod_{i=-M}^{M} g_{j_i}(S_i(f)) =~1$, we have
 $$
 \begin{aligned}
     & \mathbb{E}\Bigg[\sum_{\textbf{j}}\prod_{i=-M}^{M} g_{j_i}(S_i(f))\bigg(  \frac{y}{\delta(\log P)^{2.01}} \!\!\!\!\!\!\!\!\sum_{\substack{|k|\leqslant \delta K_P   \\ |j_k|\leqslant 1.1 \log_2 P +O(1) }} \!\!\!\!\!\!\!\!\mathbb{1}_{|S_k(f)-j_k|\leqslant 1} \,\, {\rm e}^{2j_k} \bigg)^q\Bigg]
     \\ & \ll \mathbb{E}\Bigg[\sum_{\textbf{j}}\prod_{i=-M}^{M} g_{j_i}(S_i(f))\bigg( \frac{y}{\delta(\log P)^{2.01}} \!\!\!\!\!\!\!\!\sum_{\substack{|k|\leqslant \delta K_P   \\ |j_k|\leqslant 1.1 \log_2 P +O(1) }} \!\!\!\!\bigg| F_P\bigg(\frac{ik}{ K_P }\bigg)\bigg|^2 \bigg)^q\Bigg]
     \\ & \leqslant \bigg( \frac{y}{\log P}\bigg)^q \mathbb{E}\Bigg[\bigg( \frac{1}{\delta K_P } \sum_{\substack{|k|\leqslant \delta K_P    }} \bigg| F_P\bigg(\frac{ik}{ K_P }\bigg)\bigg|^2 \bigg)^q\Bigg].
 \end{aligned}
 $$
The following lemma gives the desired bound for the above term.
 \begin{lemma}\label{discret_multiplicative_chaos}
Let $f$ be a Steinhaus multiplicative function. Let $\delta \geqslant 1$ and $P$ large, recall that $\delta = \frac{x}{y}=\log^{\theta} x $. We have uniformly for $\theta = o(\frac{1}{\sqrt{\log_2 x}})$ and $2/3 \leqslant q\leqslant 1$
     $$
\begin{aligned}
    & \bigg( \frac{y}{\log P}\bigg)^q \mathbb{E}\Bigg[\bigg( \frac{1}{\delta K_P } \sum_{\substack{|k|\leqslant \delta K_P    }} \bigg| F_P\bigg(\frac{ik}{ K_P }\bigg)\bigg|^2 \bigg)^q\Bigg]
    \\ & \ll \bigg(  y \min\bigg\{ 1, \theta \sqrt{\log_2 P} + \frac{1}{(1-q)\sqrt{\log_2 P}}\bigg\} \bigg)^q.
\end{aligned}
$$
 \end{lemma}
 \begin{proof}
Is a direct result from Proposition \ref{lemma_harper_01} and Lemma \ref{perturabtion}.
 \end{proof}
\noindent We conclude the proof by dealing with the sum over $k$ when $|S_k(f) - j_k| > 1$ in \eqref{sum_over_k_restriction}. We have by using the fact that ${\rm e }^{2j_k}\ll (\log P)^{2.2}$ for $|j_k|\leqslant 1.1 \log_2 P +O(1)$ followed by H\"{o}lder's inequality, we get then
\begin{equation}\label{expectation_last_line}
\begin{aligned}
    & \mathbb{E}\Bigg[\sum_{\textbf{j}}\prod_{i=-M}^{M} g_{j_i}(S_i(f))\bigg( \frac{y}{\delta(\log P)^{2.01}} \!\!\!\!\!\!\!\!\sum_{\substack{|k|\leqslant \delta K_P   \\ |j_k|\leqslant 1.1 \log_2 P +O(1) }} \!\!\!\!\!\!\!\!\mathbb{1}_{|S_k(f)-j_k|>1} \,\, {\rm e}^{2j_k}  \bigg)^q\Bigg]
    \\ & \ll \Bigg( \mathbb{E}\bigg[\sum_{\textbf{j}}\prod_{i=-M}^{M} g_{j_i}(S_i(f)) \frac{y}{\delta(\log P)^{2.01}} \!\!\!\!\!\!\!\!\sum_{\substack{|k|\leqslant \delta K_P   \\ |j_k|\leqslant 1.1 \log_2 P +O(1) }} \!\!\!\!\!\!\!\!\mathbb{1}_{|S_k(f)-j_k|>1} \,\, ({\log P)^{2.2}}  \bigg]\bigg)^q.
\end{aligned}
\end{equation}
Finally, by using properties \textit{2} and \textit{3} Lemma \ref{Approcimation_lemma_harper} and since $N \geqslant 1.2\log_2 P$, we have $$g_{j_k}(S_k(f))\mathbb{1}_{|j_k|\leqslant 1.1\log_2 P +O(1)}\mathbb{1}_{|S_{k}(f)-j_k|>1}\leqslant \lambda. $$ we deduce that \eqref{expectation_last_line} is
$$
\begin{aligned}
    & \leqslant \bigg(\frac{y}{\log P} \bigg)^q \Bigg(  \frac{\lambda}{\delta K_P } \sum_{\substack{|k|\leqslant \delta K_P   }}  \,\, ({\log P)^{2.2}} \mathbb{E}\bigg[\sum_{\textbf{j}}\prod_{i\neq k} g_{j_i}(S_i(f)) \bigg]\Bigg)^q
    \\ & \ll \bigg(\frac{y}{\log P} \bigg)^q \Bigg(  \frac{\lambda}{\delta K_P } \sum_{\substack{|k|\leqslant \delta K_P   }}  \,\, ({\log P)^{2.2}} \!\!\!\!\! \sum_{- N \leqslant j_k\leqslant N+1} 1\Bigg)^q \ll y^q \big(  \lambda N (\log P)^{1.2}\big)^q.
\end{aligned}
$$
By taking $\lambda =\frac{1}{N(\log P)^{1.2}\sqrt{\log_2 P}}$, we get the result.
To conclude, we verify finally that $N$ and $\lambda$ satisfy 
\begin{equation}\label{verifying}
    (x+y)P^{400(Y/\lambda )^2\log(N\log P)}<r
\end{equation}
in Lemma \ref{proposition-harper_approx_character}. Indeed recall $Y=2M+1$, $M=\lfloor (\log P)^{1.02}\rfloor$ and $P=\exp((\log L)^{1/10})$ with  $L=\min\{x+y,r/(x+y) \}+3$. We take $N=\lfloor 1.2 \log_2 P \rfloor +1$. It is easy to show that these choices verify \eqref{verifying} for $L$ large enough.
\section*{Acknowledgement}
The author would like to thank his supervisor Régis de la Bretèche for his patient guidance, encouragement and the judicious advices he has provided throughout the work that led to this paper. The author would also thank Adam Harper for his helpful remarks, corrections and useful comments on earlier versions of the paper.

\bibliographystyle{abbrv}
\bibliography{references.bib}

\begin{center}
Université Paris Cité, Sorbonne Université
CNRS,\\
Institut de Mathématiques de Jussieu- Paris Rive Gauche,\\
F-75013 Paris, France\\
E-mail: \author{rachid.caich@imj-prg.fr}
\end{center}
\end{document}